\documentclass[a4paper]{article}

\usepackage[a4paper,margin=2.5cm]{geometry}

\usepackage{amsmath,amssymb,amsfonts,amsthm}
\usepackage{mathrsfs}
\usepackage{graphicx}
\usepackage[colorlinks=true]{hyperref}

\usepackage{indentfirst}

\numberwithin{equation}{section}

\newtheorem{Thm}{Theorem}[section]
\newtheorem{Cor}{Corollary}[section]
\newtheorem{Lem}{Lemma}[section]
\newtheorem{Pro}{Proposition}[section]

\newtheorem{Rek}{Remark}[section]
\newtheorem{Def}{Definition}[section]

\newcommand{\N}{\mathbb{N}}

\newcommand{\R}{\mathbb{R}}

\title{Existence of positive and sign-changing solutions for a Choquard equation involving mixed local and nonlocal operators}

\author{Shaoxiong Chen, Hichem Hajaiej, Min Yang, Zhipeng Yang\thanks{Corresponding author: Z.~Yang.}}

\date{}  

\AtEndDocument{%
  \par
  \bigskip
  \bigskip

  \noindent
  \textbf{Shaoxiong Chen:}\\[0.2em]
  \textsc{Department of Mathematics, Yunnan Normal University, Kunming, China}\\[0.3em]
  \textit{E-mail address}: \texttt{gxmail@126.com}\\[1.5em]
    \noindent
    \textbf{Hichem Hajaiej:}\\[0.2em]
  \textsc{Department of Mathematics, California State University, Los Angeles, USA}\\[0.3em]
  \textit{E-mail address}: \texttt{hichem.hajaiej@gmail.com}\\[1.5em]
  \noindent
  \textbf{Min Yang:}\\[0.2em]
  \textsc{Department of Mathematics, Yunnan Normal University, Kunming, China}\\[0.3em]
  \textit{E-mail address}: \texttt{2563439805@qq.com}\\[1.5em]
  \noindent
  \textbf{Zhipeng Yang:}\\[0.2em]
  \textsc{Department of Mathematics, Yunnan Normal University, Kunming, China}\\
  \textsc{Yunnan Key Laboratory of Modern Analytical Mathematics and Applications, Kunming, China.}\\[0.3em]
  \textit{E-mail address}: \texttt{yangzhipeng326@163.com}%
}

\begin{document}

\maketitle

\begin{abstract}
We study the Choquard equation involving mixed local and nonlocal operators
\[
-\Delta u + (-\Delta)^{s}u + V(x)u
= \left(\frac{1}{|x|^{\mu}} * F(u)\right) f(u)
\quad \text{in } \R^{2},
\]
where $s\in(0,1)$, $\mu\in(0,2)$, $F(t)=\int_{0}^{t} f(\tau)\,d\tau$, and $f$ has subcritical exponential growth of Trudinger--Moser type. Under suitable assumptions on the potential $V$ and the nonlinearity $f$, we prove the existence of a least energy positive solution by a Nehari manifold approach. We also establish the existence of a sign-changing solution by means of invariant sets of descending flow. If, in addition, the nonlinearity is odd, then the problem admits infinitely many sign-changing solutions.

\end{abstract}

\section{Introduction and main results}

In this paper, we study the existence of a least energy positive solution and sign-changing solutions for a mixed local--nonlocal Choquard equation with a continuous potential $V$ and a nonlinearity of Trudinger--Moser subcritical growth. More precisely, we consider
\begin{equation}\label{eq1.1}
-\Delta u + (-\Delta)^{s} u + V(x) u
= \left(\frac{1}{|x|^{\mu}} * F(u)\right) f(u)
\quad \text{in } \R^{2},
\end{equation}
where $s\in(0,1)$, $\mu\in(0,2)$, $V:\R^{2}\to\R$ is continuous, $f:\R\to\R$ is continuous, and
\[
F(t)=\int_{0}^{t} f(\tau)\,d\tau.
\]
Here $\Delta$ denotes the Laplacian and $(-\Delta)^{s}$ is the fractional Laplacian which, up to a positive normalization constant, is defined by
\[
(-\Delta)^{s}u(x)=\operatorname{P.V.}\int_{\R^{2}} \frac{u(x)-u(y)}{|x-y|^{2+2s}}\,dy,
\]
where $\operatorname{P.V.}$ stands for the Cauchy principal value.

The operator in \eqref{eq1.1} combines a second-order local diffusion and a nonlocal diffusion of order $s$. For convenience, we write
\[
\mathcal{L}=-\Delta+(-\Delta)^{s}.
\]
The operator $\mathcal{L}$ has attracted considerable attention because of its relevance in models exhibiting both local and nonlocal effects. It appears, for instance, in the study of bi-modal power law distributions \cite{Pagnini2021FCAA} and in applications to optimal search theory, biomathematics, and animal foraging behavior \cite{Dipierro2023AIHPCA}. Questions concerning existence, regularity, and symmetry of solutions, as well as Faber--Krahn type inequalities, Neumann problems, and Green function estimates, have been investigated, for example, in \cite{Abatangelo2021SJMA,Arora2025PRSE,Biagi2022CPDE,Biagi2023JAM,Biagi2024CCM}.

Without the fractional Laplacian, equation \eqref{eq1.1} reduces to a classical local semilinear problem, which has been extensively studied in view of its relevance in many areas. We refer to \cite{Hajaiej2012MMM} and \cite{Hajaiej2013N} for very general settings. Consider the problem
\begin{equation}\label{eq1.2}
\left\{\begin{aligned}
-\Delta u &= f(x,u) && x \in \Omega,\\
u &= 0 && x \in \partial\Omega.
\end{aligned}\right.
\end{equation}
Without relying on symmetry assumptions, Wang \cite{Wang1991AIHP} used the linking method and Morse theory to prove the existence of positive, negative, and sign-changing solutions. Castro et al.\ \cite{Castro1997RMJM} combined direct methods with variational splitting to extend these results and established the existence of positive, negative, and sign-changing solutions under suitable assumptions on autonomous nonlinearities. Bartsch and collaborators, see \cite{Bartsch2000MZ,Bartsch1996MNA,Bartsch2003TMNA}, further advanced the theory for problem \eqref{eq1.2}: they developed an abstract critical point theory for functionals on partially ordered Hilbert spaces, used Morse index arguments to prove the existence of sign-changing solutions, and studied nodal domains together with the location of subsolutions and supersolutions.

When only the fractional Laplacian is present, extensive research has focused on the existence, multiplicity, and regularity of solutions. For Dirichlet problems involving the fractional Laplacian,
\begin{equation}\label{eq1.3}
\left\{\begin{aligned}
(-\Delta)^s u &= f(x,u) && x \in \Omega,\\
u &= 0 && x \in \mathbb{R}^N \setminus \Omega,
\end{aligned}\right.
\end{equation}
Chang and Wang \cite{Chang2014JDE}, combining the method of invariant sets of descending flow with the Caffarelli--Silvestre extension technique \cite{Caffarelli2000CPDE} and the equivalent local realization proposed by Br\"andle et al.\ \cite{Brndle2013PRSC}, proved the existence of positive, negative, and sign-changing solutions, and showed that the sign-changing solutions have exactly two nodal domains. Deng and Shuai \cite{Deng2018ADE}, also using the method of invariant sets of descending flow, established the existence of positive, negative, and sign-changing solutions to \eqref{eq1.3} under suitable assumptions, and further proved that, under an additional monotonicity condition on the nonlinearity, the least energy of a sign-changing solution is strictly greater than the ground state energy. Li et al.\ \cite{Li2017FCAA} used minimax methods and invariant sets of descending flow to prove the existence of infinitely many sign-changing solutions for a fractional Brezis--Nirenberg problem. In a series of papers, the third author and collaborators studied the dynamics, orbital stability, and normalized solutions for very general nonlinearities $f$, see \cite{Hajaiej2023CVPDE,MR3670377,MR3114821,MR3177699}.

Recently, elliptic PDEs involving mixed local and nonlocal operators have also attracted considerable attention. In \cite{Biagi2022CPDE}, Biagi et al.\ established existence results, maximum principles, and interior Sobolev regularity for the problem
\begin{equation}\label{eq1.4}
\left\{\begin{aligned}
-\Delta u+(-\Delta)^s u &= f(x) && x \in \Omega,\\
u &\ge 0 && x \in \Omega,\\
u &= 0 && x \in \mathbb{R}^N\setminus \Omega.
\end{aligned}\right.
\end{equation}
Further summability properties of solutions to \eqref{eq1.4} were obtained by Lamao et al.\ \cite{LaMao2022AM}. In \cite{Giacomoni2025NNDEA}, the authors investigated normalized solutions for a Choquard-type equation involving mixed diffusion operators,
\[
\left\{\begin{array}{l}
\mathcal{L} u+u =\mu\left(I_\alpha *|u|^p\right)|u|^{p-2} u \quad \text{in } \mathbb{R}^n, \\
\|u\|_2^2=\tau,
\end{array}\right.
\]
where $\mathcal{L}=-\Delta+\lambda(-\Delta)^s$ with $s\in(0,1)$ and $\lambda>0$. They characterized the Sobolev regularity of normalized solutions and proved the equivalence between the existence of normalized solutions and that of normalized ground states. Anthal \cite{Anthal2023JMAA} studied a mixed-operator Choquard problem on bounded domains involving the Hardy--Littlewood--Sobolev critical exponent,
\[
\left\{\begin{array}{l}
\mathcal{L} u=\left(\int_{\Omega} \frac{|u(y)|^{2_\mu^*}}{|x-y|^\mu}\,dy\right)|u|^{2_\mu^*-2} u+\lambda u^p \quad \text{in } \Omega, \\
u \equiv 0 \quad \text{in } \mathbb{R}^n \backslash \Omega,\qquad u \geq 0 \quad \text{in } \Omega,
\end{array}\right.
\]
where $\Omega\subset \mathbb{R}^n$ has $C^{1,1}$ boundary, $n\ge 3$, $0<\mu<n$, $p \in\left[1,2^*-1\right)$, $2_\mu^*=\frac{2n-\mu}{n-2}$, and $2^*=\frac{2n}{n-2}$. Huang and Hajaiej \cite{HuangHajaiej2025LazerMcKennaMixed} studied existence, uniqueness, and regularity of weak solutions for mixed Dirichlet problems driven by $-\Delta+(-\Delta)^s$ with a product-type source $h(u)f$ on bounded domains. Their analysis shows that solvability and regularity are influenced by the competition among the nonlocal term $(-\Delta)^s$, the singular behavior of $h$ near zero and at infinity, and the summability or boundary singularity of the datum $f$. For further results revealing new features of this operator, we refer to \cite{Dipierro2022AA,Garain2022NA,Hu2020CVEE,su2026explicitformulacriticalmass,su2026groundstatesolutionslocalnonlocal} and the references therein.

On the other hand, over the past two decades, much attention has been devoted to the existence, multiplicity, and qualitative properties of solutions to Choquard-type equations. The classical Choquard equation
\begin{equation}\label{eq1.5}
-\Delta u+u=\left(\int_{\mathbb{R}^3} \frac{|u(y)|^2}{|x-y|} \,dy\right) u, \quad x \in \mathbb{R}^3
\end{equation}
is also known as the Choquard--Pekar equation and arises from Pekar's polaron model \cite{Pekar1954}. It was later used by Choquard in the context of Hartree--Fock theory \cite{Lieb1997CMP}. Lieb \cite{Lieb1976SIAM} proved the existence and uniqueness of the ground state of \eqref{eq1.5} by symmetric rearrangement arguments, and Lions \cite{Lions1980NA} studied the existence of positive radial solutions and infinitely many radial solutions for related Hartree-type equations.

Recently, Chen, Yang, and Yang \cite{ChenYangYang2026MixedChoquard} analyzed a two-dimensional Choquard equation driven by a mixed diffusion operator. They introduced
\[
\mathcal L_\varepsilon u=-\varepsilon^{2}\Delta u+\varepsilon^{2s}(-\Delta)^s u,
\qquad s\in(0,1),
\]
and considered the exponentially critical mixed Choquard problem
\begin{equation}\label{eq1.6}
\mathcal L_\varepsilon u+V(x)u
=\varepsilon^{\mu-2}\bigl(I_\mu*F(u)\bigr)f(u)
\quad\text{in }\mathbb R^{2},
\qquad
I_\mu(x)=|x|^{-\mu},\ 0<\mu<2,
\end{equation}
where \(F(t)=\int_{0}^{t} f(\tau)\,d\tau\). Under a Rabinowitz-type condition on \(V\) and Trudinger--Moser type critical growth assumptions on \(f\), they proved that there exists \(\varepsilon_{0}>0\) such that for every \(\varepsilon\in(0,\varepsilon_{0})\), problem \eqref{eq1.6} admits a positive ground state solution \(u_\varepsilon\). Moreover, if \(z_\varepsilon\) is a global maximum point of \(u_\varepsilon\), then concentration occurs near the set of global minima of the potential.

These results indicate that mixed local--nonlocal diffusions in two dimensions give rise to rich concentration and ground-state phenomena under Trudinger--Moser type nonlinearities. In contrast with \cite{ChenYangYang2026MixedChoquard}, our aim here is not to study semiclassical concentration, but rather to investigate positive and sign-changing solutions for the fixed-scale equation \eqref{eq1.1}. More precisely, we prove the existence of a least energy positive solution. We also establish the existence of a sign-changing solution and, under an additional oddness assumption on the nonlinearity, infinitely many sign-changing solutions.

We impose the following assumptions on $V$ and $f$:
\begin{itemize}
\item[$(V_1)$] $V\in C(\mathbb{R}^2,\mathbb{R})$ and there exists $V_0>0$ such that
\[
V(x)\ge V_0 \quad \text{for all } x\in\mathbb{R}^2;
\]

\item[$(V_2)$] for all $M>0$, there holds
\[
\mathrm{meas}\big(\{x\in\mathbb{R}^2: V(x)\le M\}\big)<\infty;
\]

\item[$(f_1)$] $f\in C^1(\mathbb{R},\mathbb{R})$ has Trudinger--Moser subcritical growth in the sense that, for every $\alpha>0$,
\[
\lim_{|t|\to\infty}\frac{f(t)}{\exp(\alpha |t|^2)-1}=0;
\]

\item[$(f_2)$]
\[
\lim_{t \to 0} \frac{f(t)}{t}=0;
\]

\item[$(f_3)$] the map $t \mapsto \frac{f(t)}{|t|}$ is strictly increasing on $\mathbb{R}\setminus\{0\}$;

\item[$(f_4)$] there exists $\theta>1$ such that
\[
f(t)\,t \ge \theta F(t) \ge 0
\quad \text{for all } t\in\mathbb{R};
\]

\item[$(f_5)$] $f(t)=0$ for all $t \le 0$;

\item[$(f_6)$] $f(t)t>0$ for all $t\neq0$;

\item[$(f_7)$] $f$ is odd, namely
\[
f(-t)=-f(t)
\quad \text{for all } t\in\mathbb{R}.
\]
\end{itemize}

\begin{Thm}\label{Thm1.1}
Assume that $V$ satisfies $(V_1)$--$(V_2)$ and $f$ satisfies $(f_1)$--$(f_5)$.
Then \eqref{eq1.1} admits a least energy positive solution.
\end{Thm}

\begin{Rek}
In Theorem~\ref{Thm1.1}, in view of assumption $(f_5)$, the monotonicity condition in $(f_3)$ is understood only on the positive half-line. More precisely, the map
\[
t \mapsto \frac{f(t)}{t}
\]
is required to be strictly increasing on $(0,\infty)$, while no monotonicity is needed on $(-\infty,0]$, since $(f_5)$ yields
\[
f(t)=0 \qquad \text{for all } t\le 0.
\]
All uses of $(f_3)$ in the proof of Theorem~\ref{Thm1.1} are restricted to positive arguments.
\end{Rek}

\begin{Thm}\label{Thm1.2}
Assume that $V$ satisfies $(V_1)$--$(V_2)$ and $f$ satisfies $(f_1)$--$(f_4)$ and $(f_6)$. Then \eqref{eq1.1} has a sign-changing solution. If, in addition, $(f_7)$ holds, then \eqref{eq1.1} possesses infinitely many sign-changing solutions.
\end{Thm}

The paper is organized as follows. In Section~2, we set up the variational framework for \eqref{eq1.1} and establish several technical lemmas. In Section~3, we prove the existence of a least energy positive solution. In Section~4, we study sign-changing solutions and prove Theorem~\ref{Thm1.2}.

\section{The Variational Framework and Some Technical Lemmas}

For $p\in[1,+\infty)$ we write
\[
\|u\|_{p}=\Bigl(\int_{\R^{2}}|u|^{p}\,dx\Bigr)^{\frac{1}{p}},
\qquad
\|u\|_{\infty}=\operatorname*{ess\,sup}_{x\in\R^{2}}|u(x)|.
\]

The Sobolev space $H^{1}(\R^{2})$ is defined by
\[
H^{1}(\R^{2})
=\bigl\{u\in L^{2}(\R^{2}) : \nabla u\in L^{2}(\R^{2};\R^{2})\bigr\},
\]
endowed with the norm
\[
\|u\|_{H^{1}(\R^{2})}
=\Bigl(\int_{\R^{2}}\bigl(|u|^{2}+|\nabla u|^{2}\bigr)\,dx\Bigr)^{\frac12}.
\]

For $s\in(0,1)$, the fractional Sobolev space $H^{s}(\R^{2})$ is defined by
\[
H^{s}(\R^{2})
=\Bigl\{u\in L^{2}(\R^{2}) :
\int_{\R^{2}}\int_{\R^{2}}
\frac{|u(x)-u(y)|^{2}}{|x-y|^{2+2s}}\,dx\,dy<\infty\Bigr\},
\]
with norm
\[
\|u\|_{H^{s}(\R^{2})}
=\Bigl(\int_{\R^{2}}|u|^{2}\,dx
+\int_{\R^{2}}\int_{\R^{2}}
\frac{|u(x)-u(y)|^{2}}{|x-y|^{2+2s}}\,dx\,dy\Bigr)^{\frac12},
\]
and Gagliardo seminorm
\[
[u]_{s}
=\Bigl(\int_{\R^{2}}\int_{\R^{2}}
\frac{|u(x)-u(y)|^{2}}{|x-y|^{2+2s}}\,dx\,dy\Bigr)^{\frac12}.
\]

The following embedding lemma can be found in \cite{Anthal2023JMAA}.

\begin{Lem}\label{Lemma2.1}
Let $0<s<1$. Then $H^{1}(\R^{2})$ is continuously embedded into $H^{s}(\R^{2})$. More precisely, there exists a constant $C_{s}>0$ such that, for every $u\in H^{1}(\R^{2})$,
\[
[u]_{s}^{2}\le C_{s}\,\|u\|_{H^{1}(\R^{2})}^{2}
=C_{s}\bigl(\|u\|_{L^{2}(\R^{2})}^{2}+\|\nabla u\|_{L^{2}(\R^{2})}^{2}\bigr).
\]
\end{Lem}

Motivated by \eqref{eq1.1}, we introduce the Hilbert space
\[
E
=\Bigl\{u\in H^{1}(\R^{2}) : \int_{\R^{2}}V(x)|u|^{2}\,dx<+\infty\Bigr\},
\]
equipped with the inner product
\[
\begin{aligned}
\langle u,v\rangle
&=\int_{\R^{2}}\nabla u\cdot\nabla v\,dx
+\frac{1}{2}\int_{\R^{2}}\int_{\R^{2}}
\frac{(u(x)-u(y))(v(x)-v(y))}{|x-y|^{2+2s}}\,dx\,dy \\
&\quad
+\int_{\R^{2}}V(x)\,u\,v\,dx,
\end{aligned}
\]
and the associated norm
\[
\|u\|^{2}
=\int_{\R^{2}}|\nabla u|^{2}\,dx
+\frac{1}{2}\int_{\R^{2}}\int_{\R^{2}}
\frac{|u(x)-u(y)|^{2}}{|x-y|^{2+2s}}\,dx\,dy
+\int_{\R^{2}}V(x)\,|u|^{2}\,dx.
\]
For convenience, we also set
\[
\|u\|_{V}^{2}
=\int_{\R^{2}}|\nabla u|^{2}\,dx
+\int_{\R^{2}}V(x)\,|u|^{2}\,dx.
\]

By Lemma~\ref{Lemma2.1} and assumption $(V_1)$, the norms $\|\cdot\|$ and $\|\cdot\|_{V}$ are equivalent on $E$. Moreover, $(V_1)$ gives the continuous embedding $E\hookrightarrow H^1(\R^2)$. Indeed, for every $u\in E$,
\[
\begin{aligned}
\|u\|_{H^1(\R^2)}^2
&=\int_{\R^2}|\nabla u|^2\,dx+\int_{\R^2}|u|^2\,dx \\
&\le \int_{\R^2}|\nabla u|^2\,dx+\frac1{V_0}\int_{\R^2}V(x)|u|^2\,dx \\
&\le \max\{1,V_0^{-1}\}\|u\|^2.
\end{aligned}
\]
Thus $\|u\|_{H^1(\R^2)}\le C_E\|u\|$ for some $C_E>0$ and all $u\in E$. We shall use this continuous embedding below; no equivalence between $\|\cdot\|$ and the usual $H^1(\R^2)$ norm is needed.

\begin{Rek}\label{Rek2.1}
Under $(V_1)$--$(V_2)$, the space $E$ is compactly embedded into $L^{p}(\R^{2})$ for every $2\le p<+\infty$, cf.\ \cite{kondrat1999discreteness}.
\end{Rek}

For $u\in C_{c}^{\infty}(\R^{2})$, the fractional Laplacian admits a pointwise principal value representation, up to a positive normalization constant. In what follows, we choose the normalization so that, for all $u,v\in H^{s}(\R^{2})$,
\[
\int_{\R^{2}}(-\Delta)^{s}u\,v\,dx
=
\frac12\int_{\R^{2}}\int_{\R^{2}}
\frac{(u(x)-u(y))(v(x)-v(y))}{|x-y|^{2+2s}}\,dx\,dy.
\]

We next collect several analytic tools that will be used repeatedly. The following Trudinger--Moser inequality in $\R^{2}$ goes back to Cao, see \cite{Cao1992CPDE}.

\begin{Pro}\label{pro2.1}
If $\alpha>0$ and $u\in H^{1}(\R^{2})$, then
\[
\int_{\R^{2}}\bigl(\mathrm{e}^{\alpha u^{2}}-1\bigr)\,dx<+\infty.
\]
Moreover, if $\alpha<4\pi$ and $\|u\|_{2}\le M<+\infty$, then there exists a constant $C_{1}=C_{1}(M,\alpha)>0$ such that
\[
\sup_{\|\nabla u\|_{2}\le1,\;\|u\|_{2}\le M}
\int_{\R^{2}}\bigl(\mathrm{e}^{\alpha u^{2}}-1\bigr)\,dx
\le C_{1}.
\]
\end{Pro}

\begin{Lem}\cite{MR1817225}\label{Lemma2.2}
Let $t,r>1$ and $0<\mu<N$ satisfy
\[
\frac{1}{t}+\frac{\mu}{N}+\frac{1}{r}=2.
\]
If $g\in L^{t}(\R^{N})$ and $h\in L^{r}(\R^{N})$, then there exists a constant $C(t,N,\mu,r)>0$, independent of $g$ and $h$, such that
\[
\left|\int_{\R^{N}}\left(\frac{1}{|x|^{\mu}}*g\right)h\,dx\right|
\le C(t,N,\mu,r)\,\|g\|_{t}\,\|h\|_{r}.
\]
In particular, when $N=2$ and $t=r=\frac{4}{4-\mu}$, one has
\[
\int_{\R^{2}}\left(\frac{1}{|x|^{\mu}}*F(u)\right)F(u)\,dx
\le C_{\mu}\,\|F(u)\|_{\frac{4}{4-\mu}}^{2},
\]
for some constant $C_{\mu}>0$.
\end{Lem}

By $(f_1)$ and $(f_2)$, for every $\varepsilon>0$, every $\alpha>0$, and every $q>2$, there exist constants $C_{\varepsilon,\alpha,q}>0$ and $\widetilde C_{\varepsilon,\alpha,q}>0$ such that, for all $t\in\R$,
\begin{equation}\label{eq2.1}
|f(t)|
\le
\varepsilon |t|^{\frac{2-\mu}{2}}
+
C_{\varepsilon,\alpha,q}\,|t|^{q-1}\bigl[\exp(\alpha|t|^{2})-1\bigr]
\end{equation}
and
\begin{equation}\label{eq2.2}
|F(t)|
\le
\varepsilon |t|^{\frac{4-\mu}{2}}
+
\widetilde C_{\varepsilon,\alpha,q}\,|t|^{q}\bigl[\exp(\alpha|t|^{2})-1\bigr].
\end{equation}

\begin{Lem}\label{Lemma2.3}
Assume that $f$ satisfies $(f_1)$ and $(f_2)$.
Let $p_\mu=\frac{4}{4-\mu}$.
Then, for every $R>0$, there exists a constant $C_R>0$ such that
\[
\sup_{\|u\|\le R}\|F(u)\|_{L^{p_\mu}(\R^2)}\le C_R,\quad
\sup_{\|u\|\le R}\|f(u)u\|_{L^{p_\mu}(\R^2)}\le C_R,\quad
\sup_{\substack{\|u\|\le R\\ \|v\|\le 1}}
\|f(u)v\|_{L^{p_\mu}(\R^2)}
\le C_R.
\]
\end{Lem}

\begin{proof}
By the continuous embedding $E\hookrightarrow H^1(\R^2)$ established above, there exists $C_E>0$ such that
\[
\|u\|_{H^1(\R^2)}\le C_E\|u\|
\qquad\text{for all }u\in E.
\]

Fix $R>0$. Choose $\sigma>1$ and then $\alpha>0$ so small that
\[
\alpha p_\mu \sigma C_E^2R^2<4\pi.
\]
By $(f_1)$ and $(f_2)$, for every $q>2$ there exists $C_{\alpha,q}>0$ such that
\[
|F(t)|+|f(t)t|
\le
C_{\alpha,q}\Bigl(|t|^{\frac{4-\mu}{2}}+|t|^q(e^{\alpha t^2}-1)\Bigr)
\qquad\text{for all }t\in\R,
\]
and
\[
|f(t)s|
\le
C_{\alpha,q}\Bigl(|t|^{\frac{2-\mu}{2}}|s|+|t|^{q-1}(e^{\alpha t^2}-1)|s|\Bigr)
\qquad\text{for all }t,s\in\R.
\]

Let $\|u\|\le R$ and set
\[
w=\frac{u}{C_ER}.
\]
Then $\|w\|_{H^1(\R^2)}\le1$. Hence Proposition~\ref{pro2.1} yields
\[
\int_{\R^2}\bigl(e^{\alpha p_\mu\sigma u^2}-1\bigr)\,dx
=
\int_{\R^2}\Bigl(e^{\alpha p_\mu\sigma C_E^2R^2 w^2}-1\Bigr)\,dx
\le C_R,
\]
and therefore
\[
\sup_{\|u\|\le R}\|e^{\alpha u^2}-1\|_{L^{p_\mu\sigma}(\R^2)}\le C_R.
\]

Let $\sigma'$ be the conjugate exponent of $\sigma$. Using the above estimate, Hölder's inequality, and the continuous embeddings of $E$ into $L^p(\R^2)$ for all $p\ge2$, we obtain
\[
\|F(u)\|_{L^{p_\mu}}
\le
C\Bigl(
\|u\|_2^{\frac{4-\mu}{2}}
+
\|u\|_{L^{qp_\mu\sigma'}(\R^2)}^q
\|e^{\alpha u^2}-1\|_{L^{p_\mu\sigma}(\R^2)}
\Bigr)
\le C_R,
\]
and similarly
\[
\|f(u)u\|_{L^{p_\mu}}\le C_R.
\]

Now let $\|u\|\le R$ and $\|v\|\le1$. Again by Hölder's inequality and the Sobolev embeddings,
\[
\begin{aligned}
\|f(u)v\|_{L^{p_\mu}}
&\le
C\Bigl(
\||u|^{\frac{2-\mu}{2}}v\|_{L^{p_\mu}}
+
\||u|^{q-1}(e^{\alpha u^2}-1)v\|_{L^{p_\mu}}
\Bigr) \\
&\le
C\Bigl(
\|u\|_2^{\frac{2-\mu}{2}}\|v\|_2
+
\|e^{\alpha u^2}-1\|_{L^{p_\mu\sigma}(\R^2)}
\|u\|_{L^{2(q-1)p_\mu\sigma'}(\R^2)}^{q-1}
\|v\|_{L^{2p_\mu\sigma'}(\R^2)}
\Bigr)
\le C_R.
\end{aligned}
\]
This proves the desired estimates.
\end{proof}

In view of Lemmas~\ref{Lemma2.2} and \ref{Lemma2.3}, the functional
\[
I(u)
=\frac12\|u\|^{2}
-\frac12\int_{\R^{2}}\left(\frac{1}{|x|^{\mu}}*F(u)\right)F(u)\,dx,
\qquad u\in E,
\]
is well defined on $E$. Moreover, standard arguments show that
\[
I\in C^{1}(E,\R),
\]
and
\[
\begin{aligned}
I'(u)[v]
&=\int_{\R^{2}} \nabla u\cdot \nabla v\,dx
+\frac{1}{2}\int_{\R^{2}}\int_{\R^{2}}
\frac{(u(x)-u(y))(v(x)-v(y))}{|x-y|^{2+2s}}\,dx\,dy \\
&\quad
+\int_{\R^{2}}V(x)\,u\,v\,dx
-\int_{\R^{2}}\left(\frac{1}{|x|^{\mu}}*F(u)\right)f(u)\,v\,dx
\end{aligned}
\]
for all $u,v\in E$.

A function $u\in E$ is called a weak solution of \eqref{eq1.1} if
\[
I'(u)[v]=0
\qquad\text{for all } v\in E.
\]

For every $u\in E$, we write
\[
u^{+}=\max\{u,0\},
\qquad
u^{-}=\min\{u,0\},
\]
so that
\[
u=u^{+}+u^{-},
\qquad
u^{+}u^{-}=0
\quad\text{a.e. in }\R^{2}.
\]
Since the truncation maps $t\mapsto t^{\pm}$ are Lipschitz continuous, one has
\[
u^{+},u^{-}\in H^{1}(\R^{2})\cap H^{s}(\R^{2}).
\]
Moreover, from $|u^{\pm}|\le |u|$ and the definition of $E$, it follows that
\[
u^{+},u^{-}\in E.
\]

\begin{Def}\label{Def2.1}
A weak solution $u\in E$ of \eqref{eq1.1} is called sign-changing, or nodal, if
\[
u^{+}\neq 0
\qquad\text{and}\qquad
u^{-}\neq 0.
\]
\end{Def}

We now introduce the Nehari manifold
\[
\mathcal{N}
=\Bigl\{u\in E\setminus\{0\}: I'(u)[u]=0\Bigr\},
\]
and the ground-state level
\[
c=\inf_{u\in\mathcal N}I(u).
\]

\begin{Def}\label{Def2.2}
A nontrivial critical point $u\in E$ of $I$ is called a least energy weak solution of \eqref{eq1.1} if
\[
I(u)=c.
\]
If, in addition, $u>0$ in $\R^{2}$, then $u$ is called a least energy positive solution.
\end{Def}

\section{Positive solution with least energy}

\begin{Lem}\label{Lemma3.1}
Let $u\in\mathcal N$. Then
\[
I(tu)<I(u)
\qquad\text{for every }t>0\text{ with }t\ne1.
\]
In particular, if
\[
g(t)=I(tu),
\]
then
\[
g'(t)>0 \quad\text{for }0<t<1,
\qquad
g'(t)<0 \quad\text{for }t>1.
\]
\end{Lem}

\begin{proof}
Fix $u\in\mathcal N$ and define
\[
g(t)=I(tu), \qquad t>0.
\]
By $(f_5)$,
\[
F(tu)=F(tu^+),
\qquad
f(tu)u=f(tu^+)u^+
\quad\text{a.e. in }\R^2,
\]
hence
\[
g'(t)
=
t\|u\|^2-\mathcal J(t),
\]
where
\[
\mathcal J(t)
=
\int_{\R^2}\left(\frac{1}{|x|^\mu}*F(tu^+)\right)f(tu^+)u^+\,dx.
\]
Since $u\in\mathcal N$,
\[
\mathcal J(1)=\|u\|^2.
\]

We next justify the inequalities following from $(f_3)$. Let $a>0$ and set
\[
g(r)=\frac{f(r)}{r},\qquad r>0.
\]
By $(f_3)$, $g$ is strictly increasing on $(0,\infty)$. If $0<t\le1$ and $0<\tau\le a$, then $t\tau\le\tau$, and hence
\[
\frac{f(t\tau)}{t\tau}=g(t\tau)\le g(\tau)=\frac{f(\tau)}{\tau},
\]
that is,
\[
f(t\tau)\le t f(\tau).
\]
Consequently,
\[
F(ta)=t\int_0^a f(t\tau)\,d\tau
\le t^2\int_0^a f(\tau)\,d\tau
=t^2F(a).
\]
Taking $\tau=a$ also gives $f(ta)\le tf(a)$. Similarly, if $t\ge1$, then $t\tau\ge\tau$ for $\tau>0$, and the monotonicity of $g$ gives
\[
f(t\tau)\ge t f(\tau),
\]
so that
\[
F(ta)=t\int_0^a f(t\tau)\,d\tau
\ge t^2\int_0^a f(\tau)\,d\tau
=t^2F(a),
\]
and $f(ta)\ge tf(a)$. Therefore, for every $a>0$,
\[
f(ta)\le tf(a), \quad F(ta)\le t^2F(a)
\qquad\text{for }0<t\le1,
\]
and
\[
f(ta)\ge tf(a), \quad F(ta)\ge t^2F(a)
\qquad\text{for }t\ge1.
\]
Therefore,
\[
\mathcal J(t)\le t^3\mathcal J(1)=t^3\|u\|^2
\qquad\text{for }0<t<1,
\]
and
\[
\mathcal J(t)\ge t^3\mathcal J(1)=t^3\|u\|^2
\qquad\text{for }t>1.
\]
Substituting into the formula for $g'(t)$, we get
\[
g'(t)\ge t(1-t^2)\|u\|^2>0
\qquad\text{for }0<t<1,
\]
and
\[
g'(t)\le t(1-t^2)\|u\|^2<0
\qquad\text{for }t>1.
\]
Thus $t=1$ is the unique global maximizer of $g$, and consequently
\[
I(tu)<I(u)
\qquad\text{for every }t>0,\ t\ne1.
\]
\end{proof}

\begin{Lem}\label{Lemma3.2}
For each $u \in E$ with $u\ge 0$ and $u\not\equiv 0$, there exists a unique $t_u>0$ such that
\[
t_u u \in \mathcal{N}.
\]
Moreover,
\[
I(t_u u)=\max_{t>0} I(tu),
\]
and
\[
I(u)>0
\qquad\text{for every }u\in\mathcal N.
\]
\end{Lem}

\begin{proof}
Fix $u\in E$ with $u\ge0$ and $u\not\equiv0$, and set
\[
g(t)=I(tu), \qquad t>0.
\]
Then
\[
g'(t)=t\|u\|^2-\int_{\R^2}\left(\frac{1}{|x|^\mu}*F(tu)\right)f(tu)u\,dx,
\]
so that
\[
tu\in\mathcal N
\quad\Longleftrightarrow\quad
g'(t)=0.
\]

We first show that $g(t)>0$ for $t>0$ small. By \eqref{eq2.2}, Lemma~\ref{Lemma2.2}, Proposition~\ref{pro2.1}, and the continuous embeddings of $E$, for every $q>2$ there holds
\[
\int_{\R^2}\left(\frac{1}{|x|^\mu}*F(tu)\right)F(tu)\,dx
\le C\bigl(t^{4-\mu}+t^{2q}\bigr)
\]
for $t>0$ small. Hence
\[
g(t)\ge \frac{t^2}{2}\|u\|^2-C\bigl(t^{4-\mu}+t^{2q}\bigr)>0
\]
for all sufficiently small $t>0$.

Next we show that $g(t)\to-\infty$ as $t\to+\infty$. Set
\[
A(t)=\int_{\R^2}\left(\frac{1}{|x|^\mu}*F(tu)\right)F(tu)\,dx.
\]
Since $u\ge0$ and $u\not\equiv0$, one has $A(t_*)>0$ for some $t_*>0$. By $(f_4)$,
\[
A'(t)
=2\int_{\R^2}\left(\frac{1}{|x|^\mu}*F(tu)\right)f(tu)u\,dx
\ge \frac{2\theta}{t}A(t)
\]
for $t\ge t_*$. Integrating this differential inequality yields
\[
A(t)\ge A(t_*)\left(\frac{t}{t_*}\right)^{2\theta}
\qquad\text{for all }t\ge t_*.
\]
Therefore
\[
g(t)\le \frac{t^2}{2}\|u\|^2-\frac12 A(t_*)\left(\frac{t}{t_*}\right)^{2\theta}\to-\infty
\qquad\text{as }t\to+\infty,
\]
since $\theta>1$.

Thus $g$ attains its maximum at some $t_u>0$, and necessarily
\[
g'(t_u)=0,
\]
that is,
\[
t_uu\in\mathcal N.
\]

To prove uniqueness, let $t_1u,t_2u\in\mathcal N$. If $t_1\ne t_2$, then Lemma~\ref{Lemma3.1} gives
\[
I(t_2u)=I\Bigl(\frac{t_2}{t_1}(t_1u)\Bigr)<I(t_1u)
\]
and
\[
I(t_1u)=I\Bigl(\frac{t_1}{t_2}(t_2u)\Bigr)<I(t_2u),
\]
a contradiction. Hence $t_u$ is unique. The identity
\[
I(t_uu)=\max_{t>0}I(tu)
\]
follows from the construction.

Finally, if $u\in\mathcal N$, then
\[
\|u\|^2=\int_{\R^2}\left(\frac{1}{|x|^\mu}*F(u)\right)f(u)u\,dx.
\]
Using $(f_4)$, we obtain
\[
\|u\|^2\ge \theta\int_{\R^2}\left(\frac{1}{|x|^\mu}*F(u)\right)F(u)\,dx,
\]
and therefore
\[
I(u)\ge \left(\frac12-\frac{1}{2\theta}\right)\|u\|^2>0.
\]
\end{proof}

\begin{Lem}\label{Lemma3.3}
There exists a constant $C>0$ such that
\[
\|u\|^{2}\ge C
\qquad\text{for every }u\in\mathcal{N}.
\]
\end{Lem}

\begin{proof}
Assume by contradiction that there exists a sequence $\{u_n\}\subset\mathcal N$ such that
\[
\|u_n\|\to 0.
\]
By the continuous embedding $E\hookrightarrow H^1(\R^2)$,
\[
u_n\to0 \quad\text{in }H^1(\R^2).
\]
In particular,
\[
u_n\to0 \quad\text{in }L^p(\R^2)
\qquad\text{for every }2\le p<+\infty.
\]

Since $u_n\in\mathcal N$,
\[
\|u_n\|^2
=
\int_{\R^2}\left(\frac{1}{|x|^\mu}*F(u_n)\right)f(u_n)u_n\,dx.
\]
Let
\[
r=\frac{4}{4-\mu}.
\]
By Lemma~\ref{Lemma2.2},
\[
\|u_n\|^2
\le
C_\mu \|F(u_n)\|_{L^r(\R^2)}\|f(u_n)u_n\|_{L^r(\R^2)}.
\]

Fix $\alpha>0$ and $q>2$. By \eqref{eq2.1} and \eqref{eq2.2},
\[
|F(t)|+|f(t)t|
\le
C\Bigl(|t|^{\frac{4-\mu}{2}}+|t|^q(e^{\alpha t^2}-1)\Bigr)
\qquad\text{for all }t\in\R.
\]
Hence
\[
\|F(u_n)\|_{L^r}+\|f(u_n)u_n\|_{L^r}
\le
C\Bigl(
\|u_n\|_2^{\frac{4-\mu}{2}}
+\||u_n|^q(e^{\alpha u_n^2}-1)\|_{L^r}
\Bigr).
\]
Since
\[
r\frac{4-\mu}{2}=2,
\]
the first term is bounded by
\[
C\|u_n\|^{\frac{4-\mu}{2}}.
\]

For the second term, choose $\rho>1$ and let $\rho'$ be its conjugate exponent. By Hölder's inequality,
\[
\||u_n|^q(e^{\alpha u_n^2}-1)\|_{L^r}^r
\le
\left(\int_{\R^2}(e^{\alpha r\rho\,u_n^2}-1)\,dx\right)^{\frac1\rho}
\left(\int_{\R^2}|u_n|^{qr\rho'}\,dx\right)^{\frac1{\rho'}}.
\]
Since $u_n\to0$ in $H^1(\R^2)$, for $n$ large enough one has
\[
\alpha r\rho\,\|u_n\|_{H^1(\R^2)}^2<4\pi,
\]
and Proposition~\ref{pro2.1} yields
\[
\int_{\R^2}(e^{\alpha r\rho\,u_n^2}-1)\,dx\le C.
\]
Therefore
\[
\||u_n|^q(e^{\alpha u_n^2}-1)\|_{L^r}
\le
C\|u_n\|_{L^{qr\rho'}(\R^2)}^q
\le
C\|u_n\|^q.
\]

Consequently,
\[
\|F(u_n)\|_{L^r}+\|f(u_n)u_n\|_{L^r}
\le
C\Bigl(\|u_n\|^{\frac{4-\mu}{2}}+\|u_n\|^q\Bigr),
\]
and thus
\[
\|u_n\|^2
\le
C\Bigl(\|u_n\|^{4-\mu}+\|u_n\|^{2q}\Bigr).
\]
Dividing by $\|u_n\|^2$ and letting $n\to\infty$, we obtain a contradiction, since $4-\mu>2$ and $2q>2$.

Hence
\[
\inf_{u\in\mathcal N}\|u\|>0.
\]
\end{proof}

\begin{Lem}\label{Lemma3.4}
If $\left(u_n\right)\subset\mathcal{N}$ is a minimizing sequence for $c$, then $\left(u_n\right)$ is bounded in $E$.
\end{Lem}

\begin{proof}
Since $u_n\in\mathcal{N}$ and $I(u_n)\to c$, we have
\[
c+o_n(1)=I(u_n).
\]
Using $(f_4)$ and the identity $I'(u_n)[u_n]=0$, we compute
\[
\begin{aligned}
I(u_n)-\frac{1}{2\theta}I'(u_n)[u_n]
&=\left(\frac12-\frac{1}{2\theta}\right)\|u_n\|^2 \\
&\quad
+\int_{\R^2}\left(\frac{1}{|x|^\mu}*F(u_n)\right)
\left(\frac{1}{2\theta}f(u_n)u_n-\frac12F(u_n)\right)\,dx \\
&\ge \left(\frac12-\frac{1}{2\theta}\right)\|u_n\|^2.
\end{aligned}
\]
Hence
\[
c+o_n(1)\ge \left(\frac12-\frac{1}{2\theta}\right)\|u_n\|^2,
\]
which shows that $\{u_n\}$ is bounded in $E$.
\end{proof}

\begin{Lem}\label{Lemma3.5}
Let $\left(u_n\right)\subset \mathcal{N}$ be a minimizing sequence for $c$.
Then there exist a subsequence, still denoted by $\left(u_n\right)$, and a function $u\in E$ such that
\[
u_n\rightharpoonup u \quad\text{in }E,
\qquad
u_n\to u \quad\text{in }L^{p}\left(\R^{2}\right)\ \text{for every }2\le p<+\infty,
\qquad
u_n(x)\to u(x)\ \text{a.e. in }\R^{2},
\]
and
\[
\int_{\R^{2}} \left(\frac{1}{|x|^{\mu}}*F(u_{n})\right)f\left(u_n\right) u_n\,dx
\to \int_{\R^{2}} \left(\frac{1}{|x|^{\mu}}*F(u)\right) f(u)\,u\,dx,
\]
\[
\int_{\R^{2}} \left(\frac{1}{|x|^{\mu}}*F(u_{n})\right)F\left(u_n\right)\,dx
\to \int_{\R^{2}}\left(\frac{1}{|x|^{\mu}}*F(u)\right) F(u)\,dx.
\]
\end{Lem}

\begin{proof}
By Lemma~\ref{Lemma3.4}, the sequence $\{u_n\}$ is bounded in $E$. Passing to a subsequence, we may assume that
\[
u_n\rightharpoonup u \quad\text{in }E.
\]
By Remark~\ref{Rek2.1}, we also have
\[
u_n\to u \quad\text{in }L^p(\R^2)\ \text{for every }2\le p<+\infty,
\qquad
u_n(x)\to u(x)\quad\text{a.e. in }\R^2.
\]

Set
\[
p_\mu=\frac{4}{4-\mu}.
\]
Since $F$ and $t\mapsto f(t)t$ are continuous,
\[
F(u_n(x))\to F(u(x)),
\qquad
f(u_n(x))u_n(x)\to f(u(x))u(x)
\quad\text{a.e. in }\R^2.
\]

We prove that
\[
F(u_n)\to F(u)\quad\text{in }L^{p_\mu}(\R^2),
\qquad
f(u_n)u_n\to f(u)u\quad\text{in }L^{p_\mu}(\R^2).
\]
Since $\{u_n\}$ is bounded in $E$, there exists $M>0$ such that
\[
\|u_n\|\le M
\qquad\text{for all }n.
\]
By the continuous embedding $E\hookrightarrow H^1(\R^2)$, there exists $C_E>0$ such that
\[
\|v\|_{H^1(\R^2)}\le C_E\|v\|
\qquad\text{for all }v\in E.
\]
Choose $\sigma>1$ and then choose $\alpha>0$ so small that
\[
\alpha p_\mu \sigma C_E^2 M^2<4\pi.
\]
Since $(f_1)$ and $(f_2)$ hold, for some $q>2$ there exists a constant $C>0$ such that
\[
|F(t)|+|f(t)t|
\le
C\Bigl(|t|^{\frac{4-\mu}{2}}+|t|^q\bigl(e^{\alpha t^2}-1\bigr)\Bigr)
\qquad\text{for all }t\in\R.
\]

We claim that
\[
\sup_n \|F(u_n)\|_{L^{p_\mu\sigma}(\R^2)}<\infty,
\qquad
\sup_n \|f(u_n)u_n\|_{L^{p_\mu\sigma}(\R^2)}<\infty.
\]
To prove this, choose $\rho>1$ so close to $1$ that
\[
\alpha p_\mu \sigma \rho\, C_E^2 M^2<4\pi.
\]
Let $\rho'$ be the conjugate exponent of $\rho$. By Proposition~\ref{pro2.1}, the boundedness of $\{u_n\}$ in $E$, and the Sobolev embeddings of $E$ into $L^p(\R^2)$ for all finite $p$, we have
\[
\sup_n \int_{\R^2}\bigl(e^{\alpha p_\mu \sigma \rho\, u_n^2}-1\bigr)\,dx<\infty,
\]
and hence
\[
\sup_n \|e^{\alpha u_n^2}-1\|_{L^{p_\mu\sigma\rho}(\R^2)}<\infty.
\]
Therefore
\[
\|F(u_n)\|_{L^{p_\mu\sigma}}
+
\|f(u_n)u_n\|_{L^{p_\mu\sigma}}
\le C
\]
uniformly in $n$.

We now prove uniform integrability explicitly. Let $\sigma'$ be the conjugate exponent of $\sigma$. For every measurable set $A\subset\R^2$, H\"older's inequality and the preceding uniform bounds give
\[
\int_A |F(u_n)|^{p_\mu}\,dx
\le
\left(\int_A |F(u_n)|^{p_\mu\sigma}\,dx\right)^{1/\sigma}|A|^{1/\sigma'}
\le C |A|^{1/\sigma'},
\]
and similarly
\[
\int_A |f(u_n)u_n|^{p_\mu}\,dx
\le C |A|^{1/\sigma'}.
\]
Hence the families
\[
\{|F(u_n)|^{p_\mu}\}_{n\ge1},
\qquad
\{|f(u_n)u_n|^{p_\mu}\}_{n\ge1}
\]
are uniformly integrable in $L^1(\R^2)$.

It remains to check convergence in measure on the infinite-measure space $\R^2$. Since $u_n\to u$ in $L^2(\R^2)$, we have $u_n\to u$ in measure on $\R^2$. The functions $F$ and $t\mapsto f(t)t$ are continuous, and therefore
\[
F(u_n)\to F(u),
\qquad
f(u_n)u_n\to f(u)u
\quad\text{in measure on }\R^2.
\]
Vitali's theorem then yields
\[
F(u_n)\to F(u)\quad\text{in }L^{p_\mu}(\R^2),
\qquad
f(u_n)u_n\to f(u)u\quad\text{in }L^{p_\mu}(\R^2).
\]

Finally, Lemma~\ref{Lemma2.2} implies
\[
\int_{\R^{2}} \left(\frac{1}{|x|^{\mu}}*F(u_{n})\right)f\left(u_n\right) u_n\,dx
\to \int_{\R^{2}} \left(\frac{1}{|x|^{\mu}}*F(u)\right) f(u)\,u\,dx,
\]
and
\[
\int_{\R^{2}} \left(\frac{1}{|x|^{\mu}}*F(u_{n})\right)F\left(u_n\right)\,dx
\to \int_{\R^{2}}\left(\frac{1}{|x|^{\mu}}*F(u)\right) F(u)\,dx.
\]
\end{proof}

\begin{Lem}\label{Lemma3.6}
There exists $u\in\mathcal N$ such that
\[
I(u)=c.
\]
\end{Lem}

\begin{proof}
Let $(u_n)\subset\mathcal N$ be a minimizing sequence such that
\[
I(u_n)\to c.
\]
Set
\[
v_n=u_n^+.
\]
By $(f_5)$,
\[
F(u_n)=F(v_n),
\qquad
f(u_n)u_n=f(v_n)v_n
\quad\text{a.e. in }\R^2.
\]
By Lemma~\ref{Lemma3.2}, there exists $t_n>0$ such that
\[
t_nv_n\in\mathcal N.
\]
For every $t>0$, the identities in $(f_5)$ imply that the Choquard terms of $tu_n$ and $tv_n$ coincide, while the truncation property gives
\[
\|v_n\|\le \|u_n\|.
\]
Consequently,
\[
I(tv_n)\le I(tu_n)
\qquad\text{for every }t>0.
\]
Since $u_n\in\mathcal N$, Lemma~\ref{Lemma3.1} gives
\[
I(tu_n)\le I(u_n)
\qquad\text{for every }t>0.
\]
Taking $t=t_n$, we obtain
\[
I(t_nv_n)\le I(t_nu_n)\le I(u_n).
\]
Replacing $u_n$ by $t_nv_n$, we may assume that
\[
u_n\in\mathcal N,
\qquad
u_n\ge0,
\qquad
I(u_n)\to c.
\]

By Lemma~\ref{Lemma3.4}, $\{u_n\}$ is bounded in $E$. Up to a subsequence,
\[
u_n\rightharpoonup u_0 \quad\text{in }E,
\qquad
u_n(x)\to u_0(x)\quad\text{a.e. in }\R^2,
\]
with $u_0\ge0$ a.e. in $\R^2$.
We claim that $u_0\not\equiv0$. Otherwise, Lemma~\ref{Lemma3.5} yields
\[
\int_{\R^2}\left(\frac{1}{|x|^\mu}*F(u_n)\right)f(u_n)u_n\,dx\to0.
\]
Since $u_n\in\mathcal N$, it follows that
\[
\|u_n\|^2
=
\int_{\R^2}\left(\frac{1}{|x|^\mu}*F(u_n)\right)f(u_n)u_n\,dx
\to0,
\]
contradicting Lemma~\ref{Lemma3.3}. Thus $u_0\not\equiv0$.

By Lemma~\ref{Lemma3.2}, there exists a unique $t_0>0$ such that
\[
u=t_0u_0\in\mathcal N.
\]
Since $u_n\in\mathcal N$, Lemma~\ref{Lemma3.1} gives
\[
I(t_0u_n)\le I(u_n)
\qquad\text{for all }n.
\]
Moreover,
\[
t_0u_n\rightharpoonup u \quad\text{in }E,
\qquad
t_0u_n(x)\to u(x)\quad\text{a.e. in }\R^2.
\]
By the same argument as in Lemma~\ref{Lemma3.5}, applied to the bounded sequence $\{t_0u_n\}$,
\[
\int_{\R^2}\left(\frac{1}{|x|^\mu}*F(t_0u_n)\right)F(t_0u_n)\,dx
\to
\int_{\R^2}\left(\frac{1}{|x|^\mu}*F(u)\right)F(u)\,dx.
\]
Therefore, by weak lower semicontinuity of the norm,
\[
\begin{aligned}
c
&\le I(u) \\
&\le \liminf_{n\to\infty} I(t_0u_n)
\le \liminf_{n\to\infty} I(u_n)
= c.
\end{aligned}
\]
Hence
\[
I(u)=c
\qquad\text{and}\qquad
u\in\mathcal N.
\]
\end{proof}

\begin{proof}[Proof of Theorem~\ref{Thm1.1}]
By Lemma~\ref{Lemma3.6}, there exists $u\in\mathcal N$ such that
\[
I(u)=c.
\]
We first prove that $u$ is a critical point of $I$.

Assume by contradiction that $I'(u)\ne0$ in $E'$. Then there exists
$\phi\in E$ such that
\[
I'(u)[\phi]<0.
\]
Replacing $\phi$ by a positive multiple, we may assume that
\[
I'(u)[\phi]\le -2.
\]
By the continuity of $I'$, there exists $\varepsilon\in(0,1)$ such that
\begin{equation}\label{eq3.1}
I'(tu+\sigma\phi)[\phi]\le -1
\quad\text{for all } |t-1|\le \varepsilon,\ |\sigma|\le \varepsilon .
\end{equation}
Taking $\varepsilon>0$ smaller if necessary, we may also assume that
\[
\varepsilon\|\phi\|<(1-\varepsilon)\|u\|.
\]
Choose $\eta\in C^\infty([0,\infty))$ such that
\[
0\le \eta\le1,
\qquad
\eta(t)=1 \text{ for } |t-1|\le \frac{\varepsilon}{2},
\qquad
\eta(t)=0 \text{ for } |t-1|\ge \varepsilon.
\]
Define
\[
h(t)=tu+\varepsilon\eta(t)\phi,
\qquad
\Psi(t)=I(h(t)),
\qquad t>0.
\]
Then
\[
h(t)\ne0
\qquad\text{for all }t\in[1-\varepsilon,1+\varepsilon].
\]
Indeed, for such $t$,
\[
\|h(t)\|
\ge t\|u\|-\varepsilon\eta(t)\|\phi\|
\ge (1-\varepsilon)\|u\|-\varepsilon\|\phi\|>0.
\]

We now prove that
\[
\sup_{t>0}\Psi(t)<I(u).
\]
Since $u\in\mathcal N$, one has $u^+\not\equiv0$. Hence the same argument as in the proof of Lemma~\ref{Lemma3.2} gives
\[
I(tu)\to0 \quad\text{as }t\to0^+,
\qquad
I(tu)\to-\infty \quad\text{as }t\to+\infty.
\]
Since $I(u)=c>0$, there exist $0<a<1-\varepsilon$ and
$b>1+\varepsilon$ such that
\[
I(tu)\le \frac{c}{2}
\qquad\text{for }t\in(0,a]\cup[b,+\infty).
\]
On this set, $h(t)=tu$, and hence
\[
\Psi(t)\le\frac{c}{2}.
\]

It remains to consider the compact interval $[a,b]$. We claim that
\[
\Psi(t)<I(u)
\qquad\text{for all }t\in[a,b].
\]
If $|t-1|\ge\varepsilon$, then $h(t)=tu$. Since $u\in\mathcal N$,
Lemma~\ref{Lemma3.1} gives
\[
\Psi(t)=I(tu)<I(u).
\]
If $|t-1|<\varepsilon$, then by \eqref{eq3.1},
\[
\begin{aligned}
\Psi(t)
&=I(tu+\varepsilon\eta(t)\phi)  \\
&=I(tu)+\varepsilon\eta(t)
\int_0^1
I'(tu+\sigma\varepsilon\eta(t)\phi)[\phi]\\,d\sigma  \\
&\le I(tu)-\varepsilon\eta(t).
\end{aligned}
\]
If $t\ne1$, Lemma~\ref{Lemma3.1} gives $I(tu)<I(u)$, and therefore
\[
\Psi(t)<I(u).
\]
If $t=1$, then $\eta(1)=1$, and hence
\[
\Psi(1)\le I(u)-\varepsilon<I(u).
\]
Thus $\Psi(t)<I(u)$ for every $t\in[a,b]$. Since $\Psi$ is continuous and
$[a,b]$ is compact, we obtain
\[
\sup_{t\in[a,b]}\Psi(t)<I(u).
\]
Combining this estimate with the estimate on $(0,a]\cup[b,+\infty)$, we conclude that
\begin{equation}\label{eq3.2}
\sup_{t>0}\Psi(t)<I(u)=c.
\end{equation}

Define
\[
\Upsilon(t)=I'(h(t))[h(t)]
\qquad\text{for }t\in[1-\varepsilon,1+\varepsilon].
\]
Since
\[
h(1-\varepsilon)=(1-\varepsilon)u,
\qquad
h(1+\varepsilon)=(1+\varepsilon)u,
\]
we determine the signs of $\Upsilon$ at the endpoints from the fiber map.
Let
\[
g(t)=I(tu),\qquad t>0.
\]
By the proof of Lemma~\ref{Lemma3.1},
\[
g'(t)>0\quad\text{for }0<t<1,
\qquad
g'(t)<0\quad\text{for }t>1.
\]
Therefore
\[
\Upsilon(1-\varepsilon)
=
I'((1-\varepsilon)u)[(1-\varepsilon)u]
=
(1-\varepsilon)g'(1-\varepsilon)>0,
\]
and
\[
\Upsilon(1+\varepsilon)
=
I'((1+\varepsilon)u)[(1+\varepsilon)u]
=
(1+\varepsilon)g'(1+\varepsilon)<0.
\]
By the continuity of $\Upsilon$, there exists
$\bar t\in(1-\varepsilon,1+\varepsilon)$ such that
\[
\Upsilon(\bar t)=0.
\]
Since $h(\bar t)\ne0$, this means that
\[
h(\bar t)\in\mathcal N.
\]
Hence, by the definition of $c$ and by \eqref{eq3.2},
\[
c\le I(h(\bar t))
\le \sup_{t>0}\Psi(t)
<c,
\]
a contradiction. Thus
\[
I'(u)=0.
\]

We next prove that $u$ is nonnegative. Recall that
\[
u^-=\min\{u,0\}.
\]
Then $u^-\in E$. Since $I'(u)=0$, we may test the equation by $u^-$. By $(f_5)$,
\[
f(u)u^-=0
\qquad\text{a.e. in }\R^2.
\]
Therefore
\[
\begin{aligned}
0=I'(u)[u^-]
&=
\int_{\R^2}\nabla u\cdot\nabla u^-\,dx
+\frac12
\iint_{\R^2\times\R^2}
\frac{(u(x)-u(y))(u^-(x)-u^-(y))}
{|x-y|^{2+2s}}\,dx\,dy   \\
&\quad
+\int_{\R^2}V(x)u\,u^-\,dx .
\end{aligned}
\]
Since
\[
\nabla u\cdot\nabla u^-=|\nabla u^-|^2,
\qquad
u\,u^-=|u^-|^2,
\]
and
\[
(a-b)(a^- - b^-)\ge |a^- - b^-|^2
\qquad\text{for all }a,b\in\R,
\]
we obtain
\[
0=I'(u)[u^-]\ge \|u^-\|^2.
\]
Hence
\[
u^-=0,
\]
and consequently
\[
u\ge0
\qquad\text{a.e. in }\R^2.
\]

Since $u\in\mathcal N$, one has $u\not\equiv0$. Moreover, by $(f_2)$, $(f_3)$ and $(f_5)$, the function
\[
t\mapsto \frac{f(t)}{t}
\]
is strictly increasing on $(0,\infty)$ and satisfies
\[
\lim_{t\to0^+}\frac{f(t)}{t}=0.
\]
Therefore
\[
f(t)>0
\quad\text{and}\quad
F(t)>0
\qquad\text{for all }t>0.
\]
Since $u\ge0$ and $u\not\equiv0$, it follows that
\[
F(u)\not\equiv0.
\]
Hence
\[
\left(\frac1{|x|^\mu}*F(u)\right)f(u)\ge0
\qquad\text{a.e. in }\R^2.
\]
Thus $u$ is a nontrivial nonnegative weak solution of
\[
-\Delta u+(-\Delta)^s u+V(x)u
=
\left(\frac1{|x|^\mu}*F(u)\right)f(u)
\qquad\text{in }\R^2.
\]

It remains to justify the strict positivity. We first note that the right-hand side has the local integrability required for the maximum principle. Indeed, by Lemma~\ref{Lemma2.3},
\[
F(u)\in L^{p_\mu}(\R^2),
\qquad
p_\mu=\frac{4}{4-\mu}.
\]
The Hardy--Littlewood--Sobolev inequality yields
\[
\frac1{|x|^\mu}*F(u)\in L^{4/\mu}(\R^2).
\]
On the other hand, the subcritical Trudinger--Moser growth of $f$ implies that
\[
f(u)\in L^m_{\mathrm{loc}}(\R^2)
\qquad\text{for every }m<+\infty.
\]
Choosing $m>4/(4-\mu)$, we obtain
\[
\left(\frac1{|x|^\mu}*F(u)\right)f(u)
\in L^r_{\mathrm{loc}}(\R^2)
\]
for some $r>1$. Therefore the local regularity theory for mixed local--nonlocal elliptic equations gives
\[
u\in C^\alpha_{\mathrm{loc}}(\R^2)
\]
for some $\alpha\in(0,1)$. Since $u\ge0$, $u\not\equiv0$, and
\[
-\Delta u+(-\Delta)^s u+V(x)u\ge0
\qquad\text{in }\R^2
\]
in the weak sense, the strong maximum principle for mixed local--nonlocal elliptic operators, see \cite{Biagi2022CPDE}, implies that
\[
u>0
\qquad\text{in }\R^2.
\]
This completes the proof.
\end{proof}

\section{Existence and multiplicity of sign-changing solutions}

\subsection{Abstract critical point principles}

In this subsection we recall the abstract critical point framework based on invariant sets of descending flow. We shall use the formulations from \cite{LiuLiuWang2015}, which rely on the descending flow method developed in \cite{LiuSun2001}.

Let $X$ be a real Banach space and let $\Phi\in C^1(X,\R)$. Let $Y_1,Y_2\subset X$ be open sets. We set
\[
Z=Y_1\cap Y_2,
\qquad
W=Y_1\cup Y_2,
\qquad
P=\partial Y_1\cap \partial Y_2,
\]
and
\[
K=\{u\in X:\ \Phi'(u)=0\},
\qquad
K_c=\{u\in X:\ \Phi(u)=c,\ \Phi'(u)=0\},
\qquad
\Phi^c=\{u\in X:\ \Phi(u)\le c\}.
\]

\begin{Def}\label{Def4.1}
We say that $\{Y_1,Y_2\}$ is an admissible family of invariant sets with respect to $\Phi$ at level $c$ if the following holds: whenever
\[
K_c\setminus W=\varnothing,
\]
there exists $\varepsilon_0>0$ such that, for every $\varepsilon\in(0,\varepsilon_0)$, there exists a continuous mapping
\[
\sigma:X\to X
\]
satisfying
\begin{itemize}
\item[(a)] $\sigma(Y_1)\subset Y_1$ and $\sigma(Y_2)\subset Y_2$;
\item[(b)] $\sigma|_{\Phi^{c-\varepsilon}}=I$;
\item[(c)] $\sigma(\Phi^{c+\varepsilon}\setminus W)\subset \Phi^{c-\varepsilon}$.
\end{itemize}
\end{Def}

\begin{Def}\label{Def4.2}
A map $G:X\to X$ is called an isometric involution if
\[
G^2=I
\qquad\text{and}\qquad
\|Gx-Gy\|_X=\|x-y\|_X
\quad\text{for all }x,y\in X.
\]
A subset $A\subset X$ is called symmetric if
\[
Gu\in A
\qquad\text{for all }u\in A.
\]
Let
\[
\Gamma=\{A\subset X:\ A\text{ is symmetric and }0\notin A\}.
\]
For $A\in\Gamma$, the genus of $A$, denoted by $\gamma(A)$, is the smallest positive integer $n$ such that there exists a continuous map
\[
h:A\to \R^n\setminus\{0\}
\]
which is odd with respect to $G$, namely
\[
h(Gu)=-h(u)\qquad\text{for all }u\in A.
\]
If no such integer exists, we set
\[
\gamma(A)=+\infty.
\]
We also set
\[
\gamma(\varnothing)=0.
\]
\end{Def}

\begin{Def}\label{Def4.3}
Let $G:X\to X$ be an isometric involution and let $\Phi\in C^1(X,\R)$ be $G$-invariant. We say that $\{Y_1,Y_2\}$ is a $G$-admissible family of invariant sets with respect to $\Phi$ at level $c$ if there exist $\varepsilon_0>0$ and a symmetric neighborhood $N_c$ of $K_c\setminus W$ such that
\[
\gamma(N_c)<+\infty,
\]
and, for every $\varepsilon\in(0,\varepsilon_0)$, there exists a continuous mapping
\[
\sigma:X\to X
\]
satisfying
\begin{itemize}
\item[(a)] $\sigma(Y_1)\subset Y_1$ and $\sigma(Y_2)\subset Y_2$;
\item[(b)] $\sigma\circ G=G\circ \sigma$;
\item[(c)] $\sigma|_{\Phi^{c-2\varepsilon}}=I$;
\item[(d)] $\sigma(\Phi^{c+\varepsilon}\setminus (N_c\cup W))\subset \Phi^{c-\varepsilon}$.
\end{itemize}
\end{Def}

The following two abstract critical point theorems will be used later.

\begin{Thm}\label{Thm4.1}
Let $\Phi\in C^1(X,\R)$ and let $Y_1,Y_2$ be open subsets of $X$. Assume that $\{Y_1,Y_2\}$ is an admissible family of invariant sets with respect to $\Phi$ at every level
\[
c\ge c_*:=\inf_{u\in P}\Phi(u).
\]
Assume further that there exists a continuous mapping
\[
\psi:\Delta\to X
\]
such that
\begin{itemize}
\item[(a)] $\psi(\partial_1\Delta)\subset Y_1$ and $\psi(\partial_2\Delta)\subset Y_2$;
\item[(b)] $\psi(\partial_0\Delta)\cap Z=\varnothing$;
\item[(c)]
\[
\sup_{u\in\partial_0\Delta}\Phi(\psi(u))<c_*.
\]
\end{itemize}
Here
\[
\Delta=\{(t_1,t_2)\in\R^2:\ t_1,t_2\ge0,\ t_1+t_2\le1\},
\]
\[
\partial_0\Delta=\{(t_1,t_2)\in\R^2:\ t_1,t_2\ge0,\ t_1+t_2=1\},
\qquad
\partial_1\Delta=\{0\}\times[0,1],
\qquad
\partial_2\Delta=[0,1]\times\{0\}.
\]
Define
\[
\mathcal H=
\bigl\{
\varphi\in C(\Delta,X):
\varphi(\partial_1\Delta)\subset Y_1,\ 
\varphi(\partial_2\Delta)\subset Y_2,\ 
\varphi|_{\partial_0\Delta}=\psi|_{\partial_0\Delta}
\bigr\},
\]
and
\[
c_0=\inf_{\varphi\in\mathcal H}\ \sup_{u\in \varphi(\Delta)\setminus W}\Phi(u).
\]
Then
\[
c_0\ge c_*
\qquad\text{and}\qquad
K_{c_0}\setminus W\neq\varnothing.
\]
\end{Thm}

\begin{Thm}\label{Thm4.2}
Let $\Phi\in C^1(X,\R)$ be a $G$-invariant functional, where $G:X\to X$ is an isometric involution, and let $Y_1,Y_2$ be open subsets of $X$. Assume that $\{Y_1,Y_2\}$ is a $G$-admissible family of invariant sets with respect to $\Phi$ at every level
\[
c\ge c_*:=\inf_{u\in P}\Phi(u).
\]
Assume moreover that for every $n\in\N$ there exists a continuous map
\[
\varphi_n:B_{2n}\to X
\]
such that
\begin{itemize}
\item[(a)] $\varphi_n(0)\in Z$ and
\[
\varphi_n(-t)=G\varphi_n(t)
\qquad\text{for all }t\in B_{2n};
\]
\item[(b)] $\varphi_n(\partial B_{2n})\cap Z=\varnothing$;
\item[(c)]
\[
\sup_{u\in F_G\cup \varphi_n(\partial B_{2n})}\Phi(u)<c_*,
\]
where
\[
F_G=\{u\in X:\ Gu=u\},
\qquad
B_{2n}=\{t\in\R^{2n}:\ |t|\le1\}.
\]
\end{itemize}
Then there exists a sequence of critical values $\{c_j\}_{j\ge3}$ such that
\[
c_j\ge c_*,
\qquad
K_{c_j}\setminus W\neq\varnothing,
\]
and
\[
c_j\to+\infty
\qquad\text{as }j\to+\infty.
\]
\end{Thm}

\subsection{The auxiliary operator and basic estimates}

For later use, we set
\[
K=\{u\in E:\ I'(u)=0\}.
\]

For $\beta>0$, define on $E$ the bilinear form
\[
\begin{aligned}
\mathcal A_\beta(w,\varphi)
&=
\int_{\R^2}\nabla w\cdot\nabla \varphi\,dx
+\frac12\iint_{\R^2\times\R^2}
\frac{(w(x)-w(y))(\varphi(x)-\varphi(y))}{|x-y|^{2+2s}}\,dx\,dy \\
&\quad
+\int_{\R^2}(V(x)+\beta)w\varphi\,dx,
\end{aligned}
\]
and the nonlinear functional
\[
\mathcal F_\beta(u)[\varphi]
=
\int_{\R^2}\left(\frac{1}{|x|^\mu}*F(u)\right)f(u)\varphi\,dx
+\beta\int_{\R^2}u\varphi\,dx.
\]

For each $u\in E$, we define $A_\beta(u)\in E$ as the unique weak solution of
\[
\mathcal A_\beta(A_\beta(u),\varphi)=\mathcal F_\beta(u)[\varphi]
\qquad\text{for all }\varphi\in E.
\]
Equivalently, $A_\beta(u)$ is the unique weak solution of
\[
-\Delta w+(-\Delta)^s w +(V(x)+\beta)w
=
\left(\frac{1}{|x|^\mu}*F(u)\right)f(u)+\beta u
\quad\text{in }\R^2.
\]

\begin{Lem}\label{Lemma4.1}
Fix $\beta>0$. Then, for every $u\in E$, there exists a unique element $A_\beta(u)\in E$ such that
\[
\mathcal A_\beta(A_\beta(u),\varphi)=\mathcal F_\beta(u)[\varphi]
\qquad\text{for all }\varphi\in E.
\]
\end{Lem}

\begin{proof}
Fix $u\in E$. We show that $\mathcal F_\beta(u)\in E'$.

The linear term is immediate:
\[
\left|\beta\int_{\R^2}u\varphi\,dx\right|
\le C\|u\|\,\|\varphi\|.
\]

For the Choquard term, let
\[
p_\mu=\frac{4}{4-\mu}.
\]
By Lemma~\ref{Lemma2.2},
\[
\left|
\int_{\R^2}\left(\frac{1}{|x|^\mu}*F(u)\right)f(u)\varphi\,dx
\right|
\le
C_\mu \|F(u)\|_{L^{p_\mu}(\R^2)}
\|f(u)\varphi\|_{L^{p_\mu}(\R^2)}.
\]
Set
\[
R_u=\|u\|+1.
\]
By Lemma~\ref{Lemma2.3},
\[
\|F(u)\|_{L^{p_\mu}(\R^2)}\le C_{R_u}.
\]
If $\varphi=0$, there is nothing to prove. Otherwise, applying Lemma~\ref{Lemma2.3} to
\[
u\in B_{R_u}(0),
\qquad
v=\frac{\varphi}{\|\varphi\|},
\]
we get
\[
\|f(u)\varphi\|_{L^{p_\mu}(\R^2)}
=
\|\varphi\|\,
\left\|f(u)\frac{\varphi}{\|\varphi\|}\right\|_{L^{p_\mu}(\R^2)}
\le C_{R_u}\|\varphi\|.
\]
Therefore
\[
|\mathcal F_\beta(u)[\varphi]|
\le C(u)\|\varphi\|
\qquad\text{for all }\varphi\in E,
\]
and hence
\[
\mathcal F_\beta(u)\in E'.
\]

On the other hand, the bilinear form $\mathcal A_\beta$ is continuous and coercive on $E$:
\[
\mathcal A_\beta(w,w)
=
\|w\|^2+\beta\|w\|_2^2
\ge \|w\|^2
\qquad\text{for all }w\in E.
\]
Thus the Lax--Milgram theorem yields a unique $A_\beta(u)\in E$ such that
\[
\mathcal A_\beta(A_\beta(u),\varphi)=\mathcal F_\beta(u)[\varphi]
\qquad\text{for all }\varphi\in E.
\]
\end{proof}

\begin{Lem}\label{Lemma4.2}
Fix $\beta>0$. Then the map
\[
A_\beta:E\to E
\]
is continuous.
\end{Lem}

\begin{proof}
Let $u_n\to u$ in $E$. Set
\[
w_n=A_\beta(u_n),
\qquad
w=A_\beta(u),
\qquad
p_\mu=\frac{4}{4-\mu}.
\]
Since $u_n\to u$ in $E$, the sequence $\{u_n\}$ is bounded in $E$.

We first claim that
\[
\mathcal F_\beta(u_n)\to \mathcal F_\beta(u)
\qquad\text{in }E'.
\]
Arguing exactly as in the proof of Lemma~\ref{Lemma3.5}, one obtains
\[
F(u_n)\to F(u)\quad\text{in }L^{p_\mu}(\R^2).
\]
We also need the following uniform multiplier convergence:
\begin{equation}\label{eq4.multiplier-convergence}
\sup_{\|\varphi\|\le1}
\|(f(u_n)-f(u))\varphi\|_{L^{p_\mu}(\R^2)}\to0.
\end{equation}
Indeed, if \eqref{eq4.multiplier-convergence} were false, then, up to a subsequence, there would exist
$\{\varphi_n\}\subset E$ with $\|\varphi_n\|\le1$ and $\delta_0>0$ such that
\[
\|(f(u_n)-f(u))\varphi_n\|_{L^{p_\mu}(\R^2)}\ge\delta_0.
\]
Since $u_n\to u$ in $E$, we have $u_n\to u$ in measure on $\R^2$, and therefore
$f(u_n)\to f(u)$ in measure. The boundedness of $\{\varphi_n\}$ in $E$, together with Chebyshev's inequality, implies that
\[
(f(u_n)-f(u))\varphi_n\to0
\quad\text{in measure on }\R^2.
\]
Moreover, by the growth estimate \eqref{eq2.1}, the Trudinger--Moser inequality, Lemma~\ref{Lemma2.3}, and H\"older's inequality, there exists $\sigma>1$ such that
\[
\sup_n
\|(f(u_n)-f(u))\varphi_n\|_{L^{p_\mu\sigma}(\R^2)}<+\infty.
\]
Thus the family
\[
\{|(f(u_n)-f(u))\varphi_n|^{p_\mu}\}_n
\]
is uniformly integrable in $L^1(\R^2)$. Vitali's theorem gives
\[
(f(u_n)-f(u))\varphi_n\to0
\quad\text{in }L^{p_\mu}(\R^2),
\]
which is a contradiction. Hence \eqref{eq4.multiplier-convergence} holds.

For $\|\varphi\|\le1$, by Lemma~\ref{Lemma2.2} we have
\[
\begin{aligned}
&| (\mathcal F_\beta(u_n)-\mathcal F_\beta(u))[\varphi] | \\
&\le C
\|F(u_n)-F(u)\|_{L^{p_\mu}}
\|f(u_n)\varphi\|_{L^{p_\mu}}
+C
\|F(u)\|_{L^{p_\mu}}
\|(f(u_n)-f(u))\varphi\|_{L^{p_\mu}} \\
&\quad +\beta\|u_n-u\|_2\,\|\varphi\|_2.
\end{aligned}
\]
The first term tends to zero by the convergence of $F(u_n)$ and Lemma~\ref{Lemma2.3}; the second one tends to zero by \eqref{eq4.multiplier-convergence}; and the last one tends to zero because $u_n\to u$ in $E$. Therefore
\[
\|\mathcal F_\beta(u_n)-\mathcal F_\beta(u)\|_{E'}\to0.
\]

Now let
\[
z_n=w_n-w.
\]
By the definition of $A_\beta$,
\[
\mathcal A_\beta(z_n,\varphi)
=
(\mathcal F_\beta(u_n)-\mathcal F_\beta(u))[\varphi]
\qquad\text{for all }\varphi\in E.
\]
Taking $\varphi=z_n$, we obtain
\[
\|z_n\|^2
\le
\mathcal A_\beta(z_n,z_n)
=
(\mathcal F_\beta(u_n)-\mathcal F_\beta(u))[z_n]
\le
\|\mathcal F_\beta(u_n)-\mathcal F_\beta(u)\|_{E'}\,\|z_n\|.
\]
Thus
\[
\|z_n\|
\le
\|\mathcal F_\beta(u_n)-\mathcal F_\beta(u)\|_{E'}
\to0.
\]
Therefore
\[
A_\beta(u_n)\to A_\beta(u)
\qquad\text{in }E.
\]
\end{proof}

\begin{Lem}\label{Lemma4.3}
For every $u\in E$, one has
\[
\langle I'(u),u-A_\beta(u)\rangle
=
\|u-A_\beta(u)\|^2+\beta\|u-A_\beta(u)\|_2^2.
\]
In particular,
\[
\langle I'(u),u-A_\beta(u)\rangle\ge \|u-A_\beta(u)\|^2,
\]
and
\[
I'(u)=0
\qquad\Longleftrightarrow\qquad
A_\beta(u)=u.
\]
\end{Lem}

\begin{proof}
Let
\[
v=u-A_\beta(u).
\]
By the definitions of $I'(u)$ and $A_\beta(u)$,
\[
\begin{aligned}
\langle I'(u),v\rangle
&=
\langle u,v\rangle
-\int_{\R^2}\left(\frac{1}{|x|^\mu}*F(u)\right)f(u)v\,dx \\
&=
\langle u,v\rangle
+\beta\int_{\R^2}u\,v\,dx
-\mathcal A_\beta(A_\beta(u),v) \\
&=
\mathcal A_\beta(u,v)-\mathcal A_\beta(A_\beta(u),v) \\
&=
\mathcal A_\beta(u-A_\beta(u),u-A_\beta(u)),
\end{aligned}
\]
which is the desired identity. The last assertion is immediate.
\end{proof}

\begin{Cor}\label{Cor4.1}
For every $u\in E$, one has
\[
\|u-A_\beta(u)\|\le \|I'(u)\|_{E'}.
\]
In particular, if $\{u_n\}\subset E$ is a Palais--Smale sequence for $I$, then
\[
u_n-A_\beta(u_n)\to 0
\qquad\text{in }E.
\]
\end{Cor}

\begin{proof}
By Lemma~\ref{Lemma4.3},
\[
\|u-A_\beta(u)\|^2
\le
\langle I'(u),u-A_\beta(u)\rangle
\le
\|I'(u)\|_{E'}\|u-A_\beta(u)\|.
\]
This gives the first assertion. The second assertion follows immediately from the first one.
\end{proof}

\begin{Lem}\label{Lemma4.4}
The functional $I$ satisfies the $(PS)_\ell$ condition for every $\ell \in \mathbb{R}$.
\end{Lem}

\begin{proof}
Let $\{u_n\} \subset E$ be a $(PS)_\ell$ sequence, namely
\[
I(u_n)\to \ell,
\qquad
I'(u_n)\to0 \quad\text{in }E'.
\]

We first show that $\{u_n\}$ is bounded in $E$. By the definition of $I$,
\[
I(u_n)-\frac{1}{2\theta}I'(u_n)[u_n]
=
\left(\frac12-\frac{1}{2\theta}\right)\|u_n\|^2
+\frac12\int_{\R^2}\left(\frac{1}{|x|^\mu}*F(u_n)\right)
\left(\frac{1}{\theta}f(u_n)u_n-F(u_n)\right)\,dx.
\]
By $(f_4)$, the integral term is nonnegative. Hence
\[
I(u_n)-\frac{1}{2\theta}I'(u_n)[u_n]
\ge
\left(\frac12-\frac{1}{2\theta}\right)\|u_n\|^2.
\]
Since $I(u_n)\to\ell$ and $I'(u_n)\to0$ in $E'$, it follows that
\[
\ell+o(1)+o(1)\|u_n\|
\ge
\left(\frac12-\frac{1}{2\theta}\right)\|u_n\|^2,
\]
and therefore $\{u_n\}$ is bounded in $E$.

Passing to a subsequence, we may assume that
\[
u_n\rightharpoonup u \quad\text{in }E,
\]
and, by Remark~\ref{Rek2.1},
\[
u_n\to u \quad\text{in }L^p(\R^2)\ \text{for every }2\le p<+\infty,
\qquad
u_n(x)\to u(x)\quad\text{a.e. in }\R^2.
\]

Set
\[
p_\mu=\frac{4}{4-\mu}.
\]
By the convergence argument proved in Lemma~\ref{Lemma3.5}, applied to the bounded sequence $\{u_n\}$, we have
\[
F(u_n)\to F(u)\quad\text{in }L^{p_\mu}(\R^2).
\]
We now justify the second convergence. Set
\[
v_n=u_n-u.
\]
Then $\{v_n\}$ is bounded in $E$ and, since $u_n\to u$ in $L^2(\R^2)$, one has $v_n\to0$ in measure on $\R^2$. Also $f(u_n)\to f(u)$ in measure, and therefore
\[
f(u_n)v_n\to0
\quad\text{in measure on }\R^2.
\]
Choose $\sigma>1$ close to $1$. Repeating the estimate in Lemma~\ref{Lemma2.3}, with $p_\mu$ replaced by $p_\mu\sigma$ and using the boundedness of both $\{u_n\}$ and $\{v_n\}$ in $E$, gives
\[
\sup_n \|f(u_n)v_n\|_{L^{p_\mu\sigma}(\R^2)}<+\infty.
\]
Indeed, this follows from \eqref{eq2.1}, H\"older's inequality, the compact embedding of $E$ into all finite $L^p$ spaces, and the Trudinger--Moser estimate as in Lemma~\ref{Lemma2.3}. Hence, for every measurable $A\subset\R^2$,
\[
\int_A |f(u_n)v_n|^{p_\mu}\,dx
\le C |A|^{1/\sigma'},
\]
where $\sigma'$ is the conjugate exponent of $\sigma$. Thus $\{|f(u_n)v_n|^{p_\mu}\}$ is uniformly integrable in $L^1(\R^2)$. By Vitali's theorem,
\[
f(u_n)(u_n-u)=f(u_n)v_n\to0
\quad\text{in }L^{p_\mu}(\R^2).
\]
Hence, by Lemma~\ref{Lemma2.2},
\[
\int_{\R^2}\left(\frac{1}{|x|^\mu}*F(u_n)\right)f(u_n)(u_n-u)\,dx\to0.
\]

Now
\[
I'(u_n)[u_n-u]
=
\langle u_n,u_n-u\rangle
-
\int_{\R^2}\left(\frac{1}{|x|^\mu}*F(u_n)\right)f(u_n)(u_n-u)\,dx.
\]
Since
\[
\langle u_n,u_n-u\rangle
=
\|u_n-u\|^2+\langle u,u_n-u\rangle,
\]
we obtain
\[
\|u_n-u\|^2
=
I'(u_n)[u_n-u]
-\langle u,u_n-u\rangle
+
\int_{\R^2}\left(\frac{1}{|x|^\mu}*F(u_n)\right)f(u_n)(u_n-u)\,dx.
\]
The first term tends to $0$ because $I'(u_n)\to0$ in $E'$ and $\{u_n-u\}$ is bounded in $E$. The second term tends to $0$ since $u_n\rightharpoonup u$ in $E$. The last term also tends to $0$ by the previous step. Therefore
\[
\|u_n-u\|\to0.
\]

Thus every $(PS)_\ell$ sequence admits a strongly convergent subsequence in $E$, and $I$ satisfies the $(PS)_\ell$ condition for every $\ell\in\R$.
\end{proof}

\subsection{Invariant neighborhoods of the positive and negative cones}

For $\varepsilon>0$, define
\[
P_\varepsilon^+
=
\{u\in E:\ \|u^-\|<\varepsilon\},
\qquad
P_\varepsilon^-
=
\{u\in E:\ \|u^+\|<\varepsilon\}.
\]
We also denote the positive and negative cones by
\[
P^+=\{u\in E:\ u\ge0\ \text{a.e. in }\R^2\},
\qquad
P^-=\{u\in E:\ u\le0\ \text{a.e. in }\R^2\}.
\]
Let us verify that $P_\varepsilon^+$ and $P_\varepsilon^-$ are open subsets of $E$. If $u_n\to u$ in $E$, then $u_n\to u$ in $H^1(\R^2)$ and in $H^s(\R^2)$. Since the truncations $T^+(t)=t^+$ and $T^-(t)=t^-$ are normal contractions, the standard lattice property of fractional Sobolev spaces gives
\[
u_n^\pm\to u^\pm \quad\text{in }H^s(\R^2),
\]
and the usual chain rule for Lipschitz truncations gives
\[
u_n^\pm\to u^\pm \quad\text{in }H^1(\R^2).
\]
Moreover, $|u_n^\pm-u^\pm|\le |u_n-u|$ a.e.; hence
\[
\int_{\R^2}V(x)|u_n^\pm-u^\pm|^2\,dx
\le
\int_{\R^2}V(x)|u_n-u|^2\,dx\to0.
\]
Thus $u_n^\pm\to u^\pm$ in $E$, so the maps $u\mapsto\|u^\pm\|$ are continuous on $E$. Consequently, $P_\varepsilon^+$ and $P_\varepsilon^-$ are open in $E$.
We also set
\[
Y_1=P_\varepsilon^+,
\qquad
Y_2=P_\varepsilon^-,
\qquad
Z=Y_1\cap Y_2,
\qquad
W=Y_1\cup Y_2,
\qquad
P=\partial Y_1\cap \partial Y_2.
\]

\begin{Pro}\label{Pro4.1}
For every $R>0$, there exist $\beta_R>0$, $\varepsilon_0>0$, and $\kappa\in(0,1)$ such that, for every
\[
\beta\in(0,\beta_R],
\]
one has
\[
\|A_\beta(u)^-\|\le \kappa \|u^-\|
\qquad\text{whenever }\|u\|\le R,\ \|u^-\|\le \varepsilon_0,
\]
and
\[
\|A_\beta(u)^+\|\le \kappa \|u^+\|
\qquad\text{whenever }\|u\|\le R,\ \|u^+\|\le \varepsilon_0.
\]
\end{Pro}

\begin{proof}
We prove only the estimate for the negative part, since the positive one is analogous.

Fix $R>0$ and let
\[
w=A_\beta(u).
\]
Testing the equation for $w$ with $\varphi=w^-$, and using
\[
\nabla w\cdot \nabla w^- = |\nabla w^-|^2,
\qquad
(a-b)(a^- - b^-)\ge |a^- - b^-|^2,
\]
we obtain
\[
\|w^-\|^2
\le
\int_{\R^2}\left(\frac{1}{|x|^\mu}*F(u)\right)f(u)w^-\,dx
+\beta\int_{\R^2}u\,w^-\,dx.
\]
Since $w^-\le0$, $u^-\le0$, and $f(0)=0$, it follows that
\[
f(u)w^-\le f(u^-)w^-,
\qquad
uw^-\le u^-w^-,
\]
and hence
\[
\|w^-\|^2
\le
\int_{\R^2}\left(\frac{1}{|x|^\mu}*F(u)\right)f(u^-)w^-\,dx
+\beta\int_{\R^2}u^-\,w^-\,dx.
\]

The linear term satisfies
\[
\beta\int_{\R^2}u^-\,w^-\,dx
\le C\beta \|u^-\|\,\|w^-\|.
\]
Let
\[
p_\mu=\frac{4}{4-\mu}.
\]
By Lemma~\ref{Lemma2.2} and Lemma~\ref{Lemma2.3},
\[
\int_{\R^2}\left(\frac{1}{|x|^\mu}*F(u)\right)f(u^-)w^-\,dx
\le
C_R\|f(u^-)w^-\|_{L^{p_\mu}(\R^2)}
\]
whenever $\|u\|\le R$.

Fix $\delta>0$ and $q>2$. By $(f_1)$ and $(f_2)$, for $\alpha=1$ there exists $C_{\delta,q}>0$ such that
\[
|f(t)|
\le
\delta |t|+C_{\delta,q}|t|^{q-1}(e^{t^2}-1)
\qquad\text{for all }t\in\R.
\]
Therefore
\[
|f(u^-)w^-|
\le
\delta |u^-||w^-|
+
C_{\delta,q}|u^-|^{q-1}(e^{|u^-|^2}-1)|w^-|.
\]
The first term is estimated by Sobolev embedding:
\[
\||u^-|w^-\|_{L^{p_\mu}(\R^2)}
\le C\|u^-\|\,\|w^-\|.
\]

For the second term, choose $r>1$ with conjugate exponent $r'$, and choose $\varepsilon_0>0$ so small that
\[
p_\mu r C_E^2\varepsilon_0^2<4\pi,
\]
where \(C_E\) is such that
\[
\|v\|_{H^1(\R^2)}\le C_E\|v\|
\qquad\text{for all }v\in E.
\]
If $\|u^-\|\le\varepsilon_0$, Proposition~\ref{pro2.1} gives
\[
\|e^{|u^-|^2}-1\|_{L^{p_\mu r}(\R^2)}\le C,
\]
and thus, by Hölder's inequality and Sobolev embeddings,
\[
\||u^-|^{q-1}(e^{|u^-|^2}-1)w^-\|_{L^{p_\mu}}
\le
C_R\|u^-\|^{q-1}\|w^-\|.
\]
Consequently,
\[
\|f(u^-)w^-\|_{L^{p_\mu}}
\le
C_R\delta \|u^-\|\,\|w^-\|
+
C_R\|u^-\|^{q-1}\|w^-\|.
\]

Combining the above estimates, we find
\[
\|w^-\|^2
\le
C_R(\delta+\beta)\|u^-\|\,\|w^-\|
+
C_R\|u^-\|^{q-1}\|w^-\|.
\]
If $w^-\ne0$, dividing by $\|w^-\|$ gives
\[
\|w^-\|
\le
C_R(\delta+\beta)\|u^-\|
+
C_R\|u^-\|^{q-1}.
\]

Now fix $\kappa\in(0,1)$. Choose $\delta>0$ and then $\beta_R>0$ such that
\[
C_R(\delta+\beta_R)\le \frac{\kappa}{2}.
\]
Shrinking $\varepsilon_0\in(0,1)$ if necessary, we may also assume
\[
C_R\varepsilon_0^{q-2}\le \frac{\kappa}{2}.
\]
Hence, whenever $\beta\in(0,\beta_R]$, $\|u\|\le R$, and $\|u^-\|\le\varepsilon_0$, we obtain
\[
\|w^-\|
\le
\kappa\|u^-\|.
\]
Since $w=A_\beta(u)$, this proves
\[
\|A_\beta(u)^-\|\le \kappa\|u^-\|.
\]

The estimate for the positive part is proved in the same way.
\end{proof}

\begin{Cor}\label{Cor4.2}
For every $R>0$, there exist $\beta_R>0$, $\varepsilon_0>0$, and $\kappa\in(0,1)$ such that, for every
\[
\beta\in(0,\beta_R],
\qquad
\varepsilon\in(0,\varepsilon_0),
\]
one has
\[
A_\beta\bigl(\overline{P_\varepsilon^+}\cap B_R(0)\bigr)\subset P_{\kappa\varepsilon}^+,
\qquad
A_\beta\bigl(\overline{P_\varepsilon^-}\cap B_R(0)\bigr)\subset P_{\kappa\varepsilon}^-.
\]
\end{Cor}

\begin{Lem}\label{Lemma4.5}
There exists $\varepsilon_*>0$ such that, for every $\varepsilon\in(0,\varepsilon_*)$, with $Y_1=P_\varepsilon^+$, $Y_2=P_\varepsilon^-$, and $P=\partial Y_1\cap\partial Y_2$, one has
\[
c_*(\varepsilon):=\inf_{u\in P} I(u)>0.
\]
\end{Lem}

\begin{proof}
Let $u\in P=\partial Y_1\cap\partial Y_2$. Since the maps
\[
u\mapsto \|u^+\|,
\qquad
u\mapsto \|u^-\|
\]
are continuous on $E$, we have
\[
\|u^+\|=\varepsilon,
\qquad
\|u^-\|=\varepsilon.
\]
The lattice inequalities for the local and fractional parts, together with the positivity of $V$, give
\[
\|u\|^2\ge \|u^+\|^2+\|u^-\|^2.
\]
On the other hand, since $u=u^++u^-$, the triangle inequality gives
\[
\|u\|\le \|u^+\|+\|u^-\|.
\]
Therefore
\[
\sqrt{2}\,\varepsilon\le \|u\|\le 2\varepsilon.
\]

Using \eqref{eq2.2} and Lemma~\ref{Lemma2.2}, exactly as in the proof of the small-norm estimate for the Choquard term, we obtain
\[
\int_{\R^2}\left(\frac{1}{|x|^\mu}*F(u)\right)F(u)\,dx
\le C\bigl(\|u\|^{4-\mu}+\|u\|^{2q}\bigr)
\]
for some $q>2$ and some constant $C>0$ independent of $u$. Hence
\[
I(u)\ge
\frac12\|u\|^2-C\bigl(\|u\|^{4-\mu}+\|u\|^{2q}\bigr).
\]
Using
\[
\sqrt{2}\,\varepsilon\le \|u\|\le 2\varepsilon,
\]
we obtain
\[
I(u)\ge \varepsilon^2-C_1\varepsilon^{4-\mu}-C_2\varepsilon^{2q}.
\]
Choosing $\varepsilon_*>0$ sufficiently small, we get
\[
I(u)\ge \frac12\varepsilon^2
\qquad\text{for all }u\in P.
\]
Thus
\[
c_*(\varepsilon)=\inf_{u\in P}I(u)\ge \frac12\varepsilon^2>0
\]
for every $\varepsilon\in(0,\varepsilon_*)$.
\end{proof}

\begin{Lem}\label{Lemma4.6}
Let $\varepsilon\in(0,\varepsilon_*)$, where $\varepsilon_*$ is given by Lemma~\ref{Lemma4.5}. Then there exists a continuous map
\[
\psi:\Delta\to E
\]
such that
\[
\psi(\partial_1\Delta)\subset Y_1,
\qquad
\psi(\partial_2\Delta)\subset Y_2,
\qquad
\psi(\partial_0\Delta)\cap Z=\varnothing,
\]
and
\[
\sup_{u\in \psi(\partial_0\Delta)} I(u)<c_*(\varepsilon),
\]
where
\[
\Delta=\{(t_1,t_2)\in\R^2:\ t_1,t_2\ge0,\ t_1+t_2\le1\},
\]
\[
\partial_0\Delta=\{(t_1,t_2)\in\R^2:\ t_1,t_2\ge0,\ t_1+t_2=1\},
\qquad
\partial_1\Delta=\{0\}\times[0,1],
\qquad
\partial_2\Delta=[0,1]\times\{0\}.
\]
\end{Lem}

\begin{proof}
Choose $\eta^+,\eta^-\in C_c^\infty(\R^2)$ such that
\[
\eta^\pm\ge0,
\qquad
\eta^\pm\not\equiv0,
\qquad
\operatorname{supp}\eta^+\cap \operatorname{supp}\eta^-=\varnothing.
\]
For $R>0$, define
\[
\psi_R(t_1,t_2)=R(t_2\eta^+-t_1\eta^-),
\qquad (t_1,t_2)\in\Delta.
\]
Then $\psi_R$ is continuous. Moreover,
\[
\psi_R(\partial_1\Delta)\subset P^+\subset Y_1,
\qquad
\psi_R(\partial_2\Delta)\subset P^-\subset Y_2.
\]

For $(t_1,t_2)\in\partial_0\Delta$, one has
\[
(\psi_R(t_1,t_2))^+=Rt_2\eta^+,
\qquad
(\psi_R(t_1,t_2))^-= -Rt_1\eta^-.
\]
Hence
\[
\|(\psi_R(t_1,t_2))^+\|=Rt_2\|\eta^+\|,
\qquad
\|(\psi_R(t_1,t_2))^-\|=Rt_1\|\eta^-\|.
\]
Since $t_1+t_2=1$, at least one of $t_1,t_2$ is not smaller than $\frac12$. Therefore, for $R$ large enough,
\[
\psi_R(\partial_0\Delta)\cap Z=\varnothing.
\]

Next, by $(f_4)$ and $(f_6)$,
\[
F(t)>0 \quad\text{for all } t\ne0.
\]
Moreover, there exists $C_0>0$ such that
\[
F(\tau)\ge C_0|\tau|^\theta
\qquad\text{for all } |\tau|\ge1.
\]
Indeed, for $t>0$ the function $t\mapsto F(t)t^{-\theta}$ is nondecreasing, because
\[
\frac{d}{dt}\left(\frac{F(t)}{t^\theta}\right)
=
\frac{f(t)t-\theta F(t)}{t^{\theta+1}}\ge0.
\]
The same argument applied to $t\mapsto F(-t)t^{-\theta}$ on $(0,+\infty)$ gives the estimate for negative arguments.
Choose measurable sets $E_+\subset \operatorname{supp}\eta^+$ and $E_-\subset \operatorname{supp}\eta^-$ of positive measure and constants $\delta_+,\delta_->0$ such that
\[
\eta^+\ge\delta_+ \quad\text{on }E_+,
\qquad
\eta^-\ge\delta_- \quad\text{on }E_-.
\]
For $(t_1,t_2)\in\partial_0\Delta$, either $t_2\ge\frac12$ or $t_1\ge\frac12$. In the first case,
\[
t_2\eta^+-t_1\eta^-\ge \frac{\delta_+}{2}
\quad\text{on }E_+,
\]
while in the second case,
\[
t_1\eta^--t_2\eta^+\ge \frac{\delta_-}{2}
\quad\text{on }E_-.
\]
Since the supports of $\eta^+$ and $\eta^-$ are disjoint, both cases yield a measurable set $E_0$ of positive measure and a constant $\delta_0>0$, independent of $(t_1,t_2)\in\partial_0\Delta$, such that
\[
|t_2\eta^+-t_1\eta^-|\ge \delta_0
\quad\text{on }E_0.
\]
Hence, for $R$ large enough,
\[
F(\psi_R(t_1,t_2))\ge C R^\theta
\quad\text{on }E_0,
\]
with $C>0$ independent of $(t_1,t_2)\in\partial_0\Delta$. It follows that
\[
\int_{\R^2}\left(\frac{1}{|x|^\mu}*F(\psi_R(t_1,t_2))\right)F(\psi_R(t_1,t_2))\,dx
\ge c_0 R^{2\theta}
\]
for some $c_0>0$ independent of $(t_1,t_2)\in\partial_0\Delta$.

On the other hand,
\[
\|\psi_R(t_1,t_2)\|^2\le C_1R^2
\qquad\text{uniformly for }(t_1,t_2)\in\partial_0\Delta.
\]
Therefore
\[
I(\psi_R(t_1,t_2))
\le C_2R^2-C_3R^{2\theta}
\]
uniformly for $(t_1,t_2)\in\partial_0\Delta$. Since $\theta>1$,
\[
\sup_{(t_1,t_2)\in\partial_0\Delta} I(\psi_R(t_1,t_2))\to-\infty
\qquad\text{as }R\to+\infty.
\]
Thus, taking $R$ sufficiently large, we have
\[
\psi_R(\partial_0\Delta)\cap Z=\varnothing
\qquad\text{and}\qquad
\sup_{u\in\psi_R(\partial_0\Delta)} I(u)<c_*(\varepsilon).
\]
We then set
\[
\psi=\psi_R.
\]
\end{proof}

\subsection{A locally Lipschitz pseudo-gradient operator}

By Lemma~\ref{Lemma4.2}, the operator $A_\beta$ is continuous. In general, however, it is not locally Lipschitz. In order to construct a descending flow preserving the neighborhoods of the positive and negative cones, we now replace $A_\beta$ by a locally Lipschitz operator in the standard way.

\begin{Lem}\label{Lemma4.7}
Let $R>0$, and let $\beta_R>0$, $\varepsilon_0>0$, and $\kappa\in(0,1)$ be given by Corollary~\ref{Cor4.2}. Then, for every $\beta\in(0,\beta_R]$, there exists a locally Lipschitz continuous operator
\[
B_\beta:E\setminus K\to E
\]
such that
\begin{itemize}
\item[(i)] For every $\varepsilon\in(0,\varepsilon_0)$,
\[
B_\beta\bigl((\overline{P_\varepsilon^+}\cap B_R(0))\setminus K\bigr)\subset P_\varepsilon^+,
\qquad
B_\beta\bigl((\overline{P_\varepsilon^-}\cap B_R(0))\setminus K\bigr)\subset P_\varepsilon^-.
\]

\item[(ii)] For every $u\in E\setminus K$,
\[
\frac12\|u-B_\beta(u)\|
\le
\|u-A_\beta(u)\|
\le
2\|u-B_\beta(u)\|.
\]

\item[(iii)] For every $u\in E\setminus K$,
\[
\langle I'(u),u-B_\beta(u)\rangle
\ge
\frac12\|u-A_\beta(u)\|^2.
\]

\item[(iv)] If, in addition, $f$ is odd, then
\[
B_\beta(-u)=-B_\beta(u)
\qquad\text{for all }u\in E\setminus K.
\]
\end{itemize}
\end{Lem}

\begin{proof}
By Lemma~\ref{Lemma4.2}, Corollary~\ref{Cor4.1}, and Corollary~\ref{Cor4.2}, the hypotheses of the standard Lipschitz approximation result in \cite{LiuLiuWang2015,LiuSun2001} are satisfied.
We therefore obtain a locally Lipschitz map \(B_\beta\) with properties (i)--(iv).
\end{proof}

For later use, define
\[
V_\beta(u)=u-B_\beta(u),
\qquad u\in E\setminus K.
\]
Then $V_\beta$ is locally Lipschitz on $E\setminus K$, and by Lemma~\ref{Lemma4.7},
\[
\langle I'(u),V_\beta(u)\rangle\ge \frac12\|u-A_\beta(u)\|^2>0
\qquad\text{for all }u\in E\setminus K.
\]
Hence $-V_\beta$ is a pseudo-gradient vector field for $I$ on $E\setminus K$.

\subsection{Proof of Theorem~\ref{Thm1.2}}

\begin{Lem}\label{Lemma4.8}
Let $\varepsilon\in(0,\varepsilon_*)$ be fixed and set
\[
Y_1=P_\varepsilon^+,
\qquad
Y_2=P_\varepsilon^-.
\]
Then $\{Y_1,Y_2\}$ is an admissible family of invariant sets with respect to $I$ at every level $c\ge c_*(\varepsilon)$. If, in addition, $f$ is odd, then $\{Y_1,Y_2\}$ is a $G$-admissible family of invariant sets with respect to $I$ at every level $c\ge c_*(\varepsilon)$, where
\[
G(u)=-u.
\]
\end{Lem}

\begin{proof}
Fix $c\ge c_*(\varepsilon)$.

Assume first that
\[
K_c\setminus W=\varnothing.
\]
By Lemma~\ref{Lemma4.4}, the functional $I$ satisfies the $(PS)_c$ condition, so $K_c$ is compact. Choose $R>0$ such that
\[
K_c\subset B_{R/2}(0).
\]
Let $\beta\in(0,\beta_R]$, where $\beta_R$ is given by Corollary~\ref{Cor4.2}, and let $B_\beta$ be the locally Lipschitz map given by Lemma~\ref{Lemma4.7}. Define
\[
V_\beta(u)=u-B_\beta(u), \qquad u\in E\setminus K.
\]
Then $V_\beta$ is locally Lipschitz on $E\setminus K$, and by Lemma~\ref{Lemma4.7}(iii),
\[
\langle I'(u),V_\beta(u)\rangle\ge \frac12\|u-A_\beta(u)\|^2>0
\qquad\text{for all }u\in E\setminus K.
\]
Thus $-V_\beta$ is a pseudo-gradient vector field for $I$ on $E\setminus K$. Moreover, by Corollary~\ref{Cor4.2} and Lemma~\ref{Lemma4.7}(i), the sets
\[
Y_1=P_\varepsilon^+,
\qquad
Y_2=P_\varepsilon^-
\]
are positively invariant under the descending flow generated by $-V_\beta$ in $B_R(0)$.

Hence the hypotheses of the standard deformation lemma for invariant sets are satisfied, see \cite{LiuSun2001,LiuLiuWang2015}. Therefore there exists $\delta_0>0$ such that, for every $\delta\in(0,\delta_0)$, one can find a continuous map
\[
\sigma:E\to E
\]
satisfying
\[
\sigma(Y_1)\subset Y_1,
\qquad
\sigma(Y_2)\subset Y_2,
\qquad
\sigma|_{I^{\,c-\delta}}=\mathrm{Id},
\]
and
\[
\sigma\bigl(I^{\,c+\delta}\setminus W\bigr)\subset I^{\,c-\delta}.
\]
Thus $\{Y_1,Y_2\}$ is admissible at level $c$ by Definition~\ref{Def4.1}. Since $c\ge c_*(\varepsilon)$ is arbitrary, the conclusion holds for every level $c\ge c_*(\varepsilon)$.

Assume now that $f$ is odd. Then $I$ is even and
\[
G(u)=-u
\]
is an isometric involution on $E$. Again by Lemma~\ref{Lemma4.4}, the set $K_c$ is compact. Since
\[
W=Y_1\cup Y_2
\]
is symmetric, so is $K_c\setminus W$. Moreover,
\[
I(0)=0<c_*(\varepsilon)\le c,
\]
hence
\[
0\notin K_c\setminus W.
\]
If $K_c\setminus W=\varnothing$, set
\[
N_c=\varnothing.
\]
Otherwise, since $K_c\setminus W$ is compact, symmetric, and does not contain the origin, the standard finite-genus neighborhood property of the Krasnosel'skii genus yields a symmetric open neighborhood $N_c$ of $K_c\setminus W$ such that
\[
\gamma(N_c)<+\infty.
\]

By Lemma~\ref{Lemma4.7}(iv), the map $B_\beta$ can be chosen odd. Hence
\[
V_\beta(-u)=-V_\beta(u)
\qquad\text{for all }u\in E\setminus K,
\]
so the descending flow generated by $-V_\beta$ is $G$-equivariant. Therefore the equivariant deformation lemma in \cite{LiuSun2001,LiuLiuWang2015} applies. Consequently, there exists $\delta_1>0$ such that, for every $\delta\in(0,\delta_1)$, one can find a continuous map
\[
\sigma:E\to E
\]
satisfying
\[
\sigma(Y_1)\subset Y_1,
\qquad
\sigma(Y_2)\subset Y_2,
\qquad
\sigma\circ G=G\circ \sigma,
\]
\[
\sigma|_{I^{\,c-2\delta}}=\mathrm{Id},
\]
and
\[
\sigma\bigl(I^{\,c+\delta}\setminus (N_c\cup W)\bigr)\subset I^{\,c-\delta}.
\]
Thus $\{Y_1,Y_2\}$ is $G$-admissible at level $c$ by Definition~\ref{Def4.3}. Since $c\ge c_*(\varepsilon)$ is arbitrary, the conclusion follows.
\end{proof}

\begin{proof}[Proof of the existence part of Theorem~\ref{Thm1.2}]
Fix $\varepsilon\in(0,\varepsilon_*)$, where $\varepsilon_*$ is given by Lemma~\ref{Lemma4.5}, and set
\[
Y_1=P_\varepsilon^+,
\qquad
Y_2=P_\varepsilon^-,
\qquad
c_*=c_*(\varepsilon).
\]
By Lemma~\ref{Lemma4.6}, there exists a continuous map
\[
\psi:\Delta\to E
\]
such that
\[
\psi(\partial_1\Delta)\subset Y_1,
\qquad
\psi(\partial_2\Delta)\subset Y_2,
\qquad
\psi(\partial_0\Delta)\cap Z=\varnothing,
\]
and
\[
\sup_{u\in\partial_0\Delta} I(\psi(u))<c_*.
\]
Moreover, by Lemmas~\ref{Lemma4.4} and \ref{Lemma4.8}, all assumptions of Theorem~\ref{Thm4.1} are satisfied. Hence there exists
\[
u\in K_{c_0}\setminus W
\]
for some
\[
c_0\ge c_*.
\]
In particular,
\[
I'(u)=0
\qquad\text{and}\qquad
u\notin W=Y_1\cup Y_2.
\]
Since
\[
u\notin P_\varepsilon^+,
\qquad
u\notin P_\varepsilon^-,
\]
the definitions of \(P_\varepsilon^+\) and \(P_\varepsilon^-\) imply
\[
\|u^-\|\ge\varepsilon,
\qquad
\|u^+\|\ge\varepsilon.
\]
Hence
\[
u^+\neq0,
\qquad
u^-\neq0.
\]
Therefore $u$ is a sign-changing weak solution of \eqref{eq1.1}.
\end{proof}

\begin{proof}[Proof of the multiplicity part of Theorem~\ref{Thm1.2}]
Assume in addition that $(f_7)$ holds. Then $f$ is odd, hence $I$ is even. Let
\[
G(u)=-u.
\]
Then $G$ is an isometric involution on $E$, and
\[
F_G=\{u\in E:\ Gu=u\}=\{0\}.
\]

Fix $\varepsilon\in(0,\varepsilon_*)$ and set
\[
Y_1=P_\varepsilon^+,
\qquad
Y_2=P_\varepsilon^-,
\qquad
c_*=c_*(\varepsilon).
\]
By Lemmas~\ref{Lemma4.4} and \ref{Lemma4.8}, the pair $\{Y_1,Y_2\}$ is a $G$-admissible family of invariant sets with respect to $I$ at every level $c\ge c_*$.

Let $\{e_j\}_{j\ge1}\subset C_c^\infty(\R^2)$ be linearly independent sign-changing functions, and define
\[
E_n=\operatorname{span}\{e_1,\dots,e_{2n}\}.
\]
Choose an isometric linear isomorphism
\[
T_n:\R^{2n}\to E_n.
\]

We claim that there exists $R_n>0$ such that
\[
\sup_{u\in E_n,\ \|u\|=R_n} I(u)<c_*.
\]
Indeed, since $E_n$ is finite dimensional and generated by compactly supported smooth functions, there exists a bounded open set $\Omega_n\subset\R^2$ containing the support of every $u\in E_n$. Moreover, by compactness of
\[
S_n=\{u\in E_n:\ \|u\|=1\},
\]
there exist constants $\delta_n,\rho_n>0$ such that, for every $u\in S_n$, the set
\[
A_u=\{x\in\Omega_n:\ |u(x)|\ge \delta_n\}
\]
satisfies
\[
|A_u|\ge \rho_n.
\]
On the other hand, by $(f_4)$ and $(f_6)$, exactly as in the proof of Lemma~\ref{Lemma4.6}, there exists $C_0>0$ such that
\[
F(\tau)\ge C_0|\tau|^\theta
\qquad\text{for all }|\tau|\ge1.
\]
Hence, for $R\ge \delta_n^{-1}$ and $u\in S_n$,
\[
\int_{\Omega_n}F(Ru)\,dx\ge C R^\theta,
\]
with $C>0$ independent of $u\in S_n$. Since $\operatorname{supp}u\subset\Omega_n$ and $\Omega_n$ is bounded, the positivity of the kernel gives
\[
\int_{\R^2}\left(\frac{1}{|x|^\mu}*F(Ru)\right)F(Ru)\,dx\ge C_n R^{2\theta}
\]
for some $C_n>0$ independent of $u\in S_n$. Therefore
\[
I(Ru)\le \frac12 R^2-C_nR^{2\theta}
\qquad\text{for all }u\in S_n.
\]
Since $\theta>1$,
\[
\sup_{u\in S_n}I(Ru)\to-\infty
\qquad\text{as }R\to+\infty.
\]
This proves the claim.

Choose $R_n>2\varepsilon$ such that
\[
\sup_{u\in E_n,\ \|u\|=R_n} I(u)<c_*.
\]
Define
\[
\varphi_n(t)=R_nT_n(t),
\qquad
t\in B_{2n},
\]
where
\[
B_{2n}=\{t\in\R^{2n}:\ |t|\le1\}.
\]
Then
\[
\varphi_n(0)=0\in Z,
\qquad
\varphi_n(-t)=G\varphi_n(t)
\qquad\text{for all }t\in B_{2n}.
\]
Moreover, if $t\in\partial B_{2n}$, then
\[
\|\varphi_n(t)\|=R_n>2\varepsilon.
\]
Since every $u\in Z=P_\varepsilon^+\cap P_\varepsilon^-$ satisfies
\[
\|u\|\le \|u^+\|+\|u^-\|<2\varepsilon,
\]
we have
\[
\varphi_n(\partial B_{2n})\cap Z=\varnothing.
\]
Also,
\[
\sup_{u\in \varphi_n(\partial B_{2n})}I(u)<c_*,
\qquad
I(0)=0<c_*,
\]
and hence
\[
\sup_{u\in F_G\cup \varphi_n(\partial B_{2n})}I(u)<c_*.
\]

Therefore all the assumptions of Theorem~\ref{Thm4.2} are satisfied. Applying it, we obtain a sequence of critical values $\{c_j\}_{j\ge3}$ such that
\[
c_j\ge c_*,
\qquad
K_{c_j}\setminus W\neq\varnothing,
\qquad
c_j\to+\infty
\quad\text{as }j\to+\infty.
\]
For each $j$, choose
\[
u_j\in K_{c_j}\setminus W.
\]
Then
\[
I'(u_j)=0,
\qquad
u_j\notin W,
\]
so
\[
u_j^+\neq0,
\qquad
u_j^-\neq0.
\]
Therefore each $u_j$ is a sign-changing weak solution of \eqref{eq1.1}.
\end{proof}

\section*{Acknowledgments}
We would like to thank the anonymous referee for his/her careful readings of our manuscript and the useful comments. 

\medskip
{\bf Funding:} This work is supported by National Natural Science Foundation of China (12301145, 12261107, 12561020) and Yunnan Fundamental Research Projects (202301AU070144, 202401AU070123).

\medskip
{\bf Author Contributions:} All the authors wrote the main manuscript text together and these authors contributed equally to this work.

\medskip
{\bf Data availability:}  Data sharing is not applicable to this article as no new data were created or analyzed in this study.

\medskip
{\bf Conflict of Interests:} The authors declares that there is no conflict of interest.


\end{document}